\documentclass[english,draft,reqno]{amsart}   
\usepackage[T1]{fontenc}
\usepackage[latin1]{inputenc}
\usepackage[english]{babel}
\usepackage{amsmath,amsfonts,amssymb,amsthm}
\usepackage{mathrsfs}
\usepackage{graphicx}
\usepackage[all]{xy}
\usepackage{xcolor}

	\advance \topmargin by -\headheight
	\advance \topmargin by -\headsep

	\evensidemargin \oddsidemargin
	\newcommand{\ftn}[3]{ #1 : #2 \rightarrow #3 }
		\newcommand{\setof}[2]{\ensuremath{\left\{ #1 \: : \: #2 \right\}}}

	\newcommand{\id}{\ensuremath{\operatorname{id}}}

		\newcommand{\Z}{\ensuremath{\mathbb{Z}}}
	\newcommand{\C}{\ensuremath{\mathbb{C}}}

	\newcommand{\N}{\ensuremath{\mathbb{N}}}

	\theoremstyle{plain}
	\newtheorem{thm}{Theorem}[section]
	\newtheorem{lemma}[thm]{Lemma}
	\newtheorem{theorem}[thm]{Theorem}
	\newtheorem{proposition}[thm]{Proposition}
	\newtheorem{corollary}[thm]{Corollary}

	\theoremstyle{definition}
	\newtheorem{definition}[thm]{Definition}
	
	\newtheorem{remark}[thm]{Remark}

	\numberwithin{equation}{section}
	\numberwithin{figure}{section}

\begin{document}
	\title{Non-stable K-theory for Leavitt path algebras}

	\author[D.~Hay]{Damon Hay}
	\address{Department of Mathematics\\Sam Houston State University\\
Hunstville, Texas\\
77341, USA}
        \email{dhay@shsu.edu}

	\author[M.~Loving]{Marissa Loving}
        \address{Department of Mathematics\\University of Hawaii,
Hilo\\200 W. Kawili St.\\
Hilo, Hawaii\\
96720 USA}
	\email{loving7@hawaii.edu}

	  \author[M.~Montgomery]{Martin Montgomery}
        \address{Department of Mathematics\\University of Kentucky\\
Lexington, KY\\
40506, USA}
        \email{martin.montgomery@uky.edu}

	\author[E.~Ruiz]{Efren Ruiz}
        \address{Department of Mathematics\\University of Hawaii,
	Hilo\\200 W. Kawili St.\\
Hilo, Hawaii\\
96720 USA}
\email{ruize@hawaii.edu}

\author[K.~Todd]{Katherine Todd}
        \address{Department of Mathematics\\Brandeis University \\ Goldsmith 218, MS 050 \\ 415 South Street \\ Waltham, Massachusetts \\ 
02453 USA}
        \email{ktodd@brandeis.edu}

        \date{\today}
	


	\begin{abstract}
	We compute the monoid $\mathcal{V} [ L_{K} (E) ]$ of isomorphism classes of finitely generated projective modules of a Leavitt path algebra over an arbitrary directed graph.  Our result generalizes the result of Ara, Moreno, and Pardo in which they computed the monoid $\mathcal{V} [ L_{K} (E) ]$ of a Leavitt path algebra over a countable row-finite directed graph.
	\end{abstract}

        \maketitle

\section{Introduction}

In \cite{gaga:leavittpath1}, G.~Abrams and G.~Aranda Pino introduced a class of algebras over a field $K$, which they constructed from directed graphs called \emph{Leavitt path algebras}.  The definition in \cite{gaga:leavittpath1} was for countable row-finite directed graphs, but they later extended the definition in \cite{gaga:leavittpath2} to all countable directed graphs.  K.R.~Goodearl in \cite{kg:leavittlimits} extended the notion of Leavitt path algebras $L_{K} (E)$ to all (possibly uncountable) directed graphs $E$.  Leavitt path algebras are generalization of the Leavitt algebras $L(1,n)$ of \cite{wl:modtype} and also contain many interesting classes of algebras (see \cite{gaga:leavittpath1} and \cite{gaga:leavittpath2}).  Moreover, there are nice relationships between the class of Leavitt path algebras and their analytic counterparts, graph $C^{*}$-algebras (see \cite{raeburn} for the definition of graph $C^{*}$-algebras).  In particular, M.~Tomforde showed in \cite{mt:idealst} that for any countable directed graph $E$, we have that $L_{\C} (E)$ is $*$-isomorphic to a dense sub-algebra of $C^{*} (E)$.  Also, G.~Abrams and M.~Tomforde showed in \cite{gamt:isomorita} that for countable directed graphs $E$ and $F$ with no cycles, $L_{\C} (E)$ and $L_{\C} (F)$ are Morita equivalent if and only if $C^{*} (E)$ and $C^{*} (F)$ are strongly Morita equivalent.

The monoid $\mathcal{V} [A]$ of a $C^{*}$-algebra $A$ or a $K$-algebra has played an extremely important role in the theories of both structures.  For example, $\mathcal{V} [A]$ can be used to classify direct limits of finite dimensional $C^{*}$-algebras or classify direct limits of finite dimensional $\C$-algebras (see \cite{af}).  Structural properties of the algebra and its projective modules are reflected in the structure of $\mathcal{V}[A]$.  This machinery has become known as \emph{non-stable $K$-theory}  (see \cite{bb:non-stable-K-theory}, \cite{agop:exhangerings}, and \cite{paaf-direct-sum}).  Other important properties of $A$ such as the lattice of (closed) two-sided  ideals of $A$ is encoded in $\mathcal{V} [A]$ (see Proposition~IV.5.1 of \cite{DavidsonBook}, Theorem~2.1 of \cite{affh-Projective-Modules}, and Theorem~5.3 of \cite{amp:nonstablekthy}).

In \cite{amp:nonstablekthy}, P.~Ara, M.A.~Moreno, and E.~Pardo computed the monoid $\mathcal{V}[L_K(E) ]$ of isomorphism classes of finitely generated projective modules over $L_{K} (E)$ associated with a countable row-finite directed graph $E$. They showed that $\mathcal{V}[L_K(E)]$ is naturally isomorphic to a universal abelian monoid $M_E$.  The monoid $M_{E}$ is isomorphic to $F_{E} /\sim$, where $F_{E}$ is the free abelian monoid on the set of vertices of $E$, and $\sim$ is a certain congruence on $F_{E}$ defined by $E$.  A consequence of their result is that the natural inclusion $L_\C(E) \rightarrow C^*(E)$ induces a monoid isomorphism $\mathcal{V}[L_\C(E)] \rightarrow \mathcal{V}[C^*(E)]$.  As a result, there is a natural isomorphism between the monoid $M_E$ and $\mathcal{V}[C^*(E)]$.  Their result together with the $K$-theory computation of $C^{*} (E)$ given in \cite{irws-Cuntz-Krieger} completely describes the ordered $K$-theory of $C^{*} (E)$ for a countable row-finite directed graph $E$.  Another consequence of this isomorphism is that $C^*(E)$ has stable weak cancellation or, equivalently, $C^*(E)$ is separative.  The fourth named author together with S.~Eilers and G.~Restorff in \cite{err:fullext} used this fact to give a $K$-theoretical description of when an extension of a graph $C^{*}$-algebra is a full extension.  Also, together with S.~Arklint, the fourth named author in \cite{ar-corners-ck-algs} used the monoid isomorphism $M_{E} \cong \mathcal{V} [ C^{*} (E) ]$ to prove permanence properties for graph $C^{*}$-algebras associated to finite graphs.

The objective of this paper is to compute $\mathcal{V}[L_K(E)]$ for an arbitrary directed graph $E$.  A consequence of this computation is that the natural inclusion $L_\C(E) \rightarrow C^*(E)$ induces a monoid isomorphism $\mathcal{V}[L_\C(E)] \rightarrow \mathcal{V}[C^*(E)]$ for every countable directed graph $E$.  Following Ara, Moreno, and Pardo in \cite{amp:nonstablekthy}, we define a universal abelian monoid $\overline{M}_{E}$ and prove that $\overline{M}_{E}$ is naturally isomorphic to $\mathcal{V} [ L_{K} (E) ]$.  The monoid $\overline{M}_{E}$ is defined as follows:  Let $E^{0}$ be the set of vertices of $E$ and let $\mathcal{S}$ be the set consisting of $a_{v, S}$, one for each infinite emitter $v \in E^{0}$ and finite non-empty subset $S$ of edges with source $v$.  Then $\overline{M}_{E} = F_{ E^{0} \cup \mathcal{S} } / \sim$, where $F_{ E^{0} \cup \mathcal{S} }$ is the free abelian monoid on $E^{0} \cup \mathcal{S}$  and $\sim$ is a certain congruence on $F_{ E^{0} \cup \mathcal{S} }$ defined by $E$.  In the case that $E$ is row-finite, we have that $M_{E} = \overline{M}_{E}$.  The monoid isomorphism $\overline{M}_{E} \cong \mathcal{V} [ C^{*} (E) ]$ can be used to reprove the results in \cite{mt:orderkthy}, where the ordered $K_{0}$-group of $C^{*} (E)$ was computed for a countable directed graph $E$.  The authors have recently been informed by S.~Eilers and T.~Katsura that they have independently proved Theorem~\ref{t:iso-countable} in their study of semi-projective graph $C^{*}$-algebras \cite{setk:semiproj}.  We expect that the monoid isomorphisms  $\overline{M}_{E} \cong \mathcal{V} [ C^{*} (E) ] \cong \mathcal{V} [ L_{K} (E) ]$ can be used to study other structural properties of $C^{*}(E)$ and $L_{K} (E)$. 

The main result of the paper, $\overline{M}_{E}$ is naturally isomorphic to $\mathcal{V} [ L_{E} ( E ) ]$ for an arbitrary directed graph $E$, is proved in two steps.  The first step is to prove $\overline{M}_{E}$ is naturally isomorphic to $\mathcal{V}[L_K(E)]$ in the case that $E$ is a countable directed graph.  This is done by reducing the problem to the row-finite case using the Drinen-Tomforde Disingularization of $E$.  The next step is to use the fact that every directed graph is the limit of countable directed graphs to reduce the general case to the countable case.  Along the way, it is shown that $\overline{M}_{E}$ is a continuous functor from $\textbf{CKGr}$ to $\textbf{CMon}_{0}$.  Here $\textbf{CKGr}$ is the category whose objects are directed graphs and whose set of morphisms are CK-morphisms (as defined in \cite{kg:leavittlimits}) and $\textbf{CMon}_{0}$ is the category whose objects are abelian monoids and whose set of morphisms are monoid homomorphisms that preserve the identity element.      
 
\section{Definitions}\label{definitions}

A (directed) graph $E=(E^0, E^1, r_{E}, s_{E})$ consists of a set
$E^0$ of vertices, a set $E^1$ of edges, and maps $r_{E},s_{E}: E^1
\rightarrow E^0$ identifying the range and source of each edge.  A graph $E$ is \emph{countable} if $E^{0}$ and $E^{1}$ are countable sets.  A
vertex $v \in E^0$ is called a \emph{sink} if $|s_{E}^{-1}( v )|=0$, and $v$
is called an \emph{infinite emitter} if $|s_{E}^{-1}(v )|=\infty$. A graph
$E$ is said to be \emph{row-finite} if it has no infinite emitters. If
$v$ is either a sink or an infinite emitter, then we call $v$ a
\emph{singular vertex}.  We write $E^0_\textnormal{sing}$ for the set
of singular vertices.  Vertices that are not singular vertices are
called \emph{regular vertices} and we write $E^0_\textnormal{reg}$ for
the set of regular vertices.  

\subsection{Leavitt path algebras}
 
We now recall the definition of Leavitt path algebras given in \cite{kg:leavittlimits}.  Let $E$ be a graph.  The \emph{path algebra of $E$ over $K$}, denoted by $KE$, is the $K$-algebra based on the vector space over $K$ with basis the set of all paths in $E$, and with multiplication induced by concatenation of paths: $p = e_{1} e_{2} \cdots e_{n}$ and $q = f_{1} f_{2} \cdots f_{m}$ are paths in $E$, their product in $KE$ is given by 
\begin{align*}
pq = 
\begin{cases}
e_{1} e_{2} \cdots e_{n} f_{1} f_{2} \cdots f_{m}, &\text{if $r_{E} ( e_{n} ) = s_{E} ( f_{1} )$} \\
0,		&\text{otherwise}.
\end{cases}
\end{align*} 
Note that $KE$ is the $K$-algebra presented by generators from $E^{0} \sqcup E^{1}$ with the following relations
\begin{itemize}
\item[(1)] $v^{2} = v$ for all $v \in E^{0}$; 

\item[(2)] $vw = \delta_{v,w} v$ for all $v ,w \in E^{0}$; and 

\item[(3)] $s_{E}(e)e = er_{E}(e) = e$ for all $e \in E^{1}$.
\end{itemize}

Let $E$ be a graph.  The \emph{dual graph of $E$} is a graph $E^{*}$ consisting of the same vertices as $E$ but with all edges reverse.  Then the \emph{double} or \emph{(extended graph)} of $E$, denoted by $\widehat{E}$, is the union of $E$ and $E^{*}$.

\begin{definition}[Definition~1.4 of \cite{kg:leavittlimits}]
Let $E = ( E^{0} , E^{1} , r_{E} , s_{E} )$ be a graph and let $K$ be a field.  The \emph{Leavitt path algebra of $E$ over $K$}, denoted by $L_{K} (E)$, is the quotient of $K \widehat{E}$ modulo the ideal generated by the following elements:
\begin{itemize}
\item[(1)] $e^{*} e - r_{E} (e)$ for all $e \in E^{1}$;

\item[(2)] $e^{*} f$ for all distinct $e, f \in E^{1}$; and 

\item[(3)] $v - \sum_{ e \in s_{E}^{-1} (  v  ) } ee^{*}$ whenever $v \in E_{ \mathrm{reg} }^{0}$.
\end{itemize}  
\end{definition}

The above definition coincides with the definition given in \cite{gaga:leavittpath1} (see Definition~1.3 of  \cite{gaga:leavittpath1}) for countable row-finite graphs.    
 
\begin{definition}[Definition~2.4 of \cite{mt:leavitt-commutative}]
Let $E = ( E^{0} , E^{1} , r_{E} , s_{E} )$ be a graph and let $R$ be a ring.  A collection $\setof{ P_{v}, S_{e}, S_{e^{*}} }{ v \in E^{0}, e\in E^{1} } \subseteq R$ is a \emph{Leavitt $E$-family in $R$} if $\setof{ P_{v} }{ v \in E^{0} }$ consists of pairwise orthogonal idempotents and the following conditions are satisfied: 
\begin{itemize}
\item[(1)] $P_{s_{E}(e)}S_{e} = S_{e}P_{ r_{E}(e) } = S_{e}$ for all $e \in E^{1}$; 
\item[(2)] $P_{r_{E}(e) } S_{e^{*}} = S_{e^{*}} P_{ s_{E}(e)} = S_{e^{*}}$ for all $e \in E^{1}$; 
\item[(3)] $S_{e^{*}} S_{ f } = \delta_{e,f} P_{ r_{E}(e) }$ for all $e, f \in E^{1}$; and 
\item[(4)] $P_{v} = \sum_{ e \in s_{E}^{-1} (  v  ) } S_{e}S_{e^{*}}$ whenever $v \in E_{ \mathrm{reg} }^{0}$. 
\end{itemize}
\end{definition} 
  
\begin{remark}
Let $E$ be a graph.  Note that if $A$ is a $K$-algebra and 
\begin{align*}
\setof{ P_{v}, S_{e}, S_{e^{*}} }{ v \in E^{0}, e\in E^{1} }
\end{align*}
is a Leavitt $E$-family in $A$, then there exists a $K$-algebra homomorphism $\ftn{ \phi }{ L_{K} (E) }{A}$ such that 
\begin{align*}
\phi ( v ) = P_{v}, \quad \phi ( e ) = S_{e} , \quad \phi ( e^{*} ) = S_{e^{*}}.
\end{align*}
Hence, $L_{K} (E)$ is the universal $K$-algebra generated by a Leavitt $E$-family.  Therefore, the definition of Leavitt path algebras for arbitrary graphs given here coincides with the definition given in \cite{gaga:leavittpath2} (see Definition~1.1 of  \cite{gaga:leavittpath2}) for countable graphs.
\end{remark} 
 
\begin{remark}
Let $E$ be a graph and let $K$ be a field.  We will denote the generators of $L_{K} (E)$ by
\begin{align*}
\setof{ p_{v} , t_{e} }{ v \in E^{0} , e \in E^{1} }.
\end{align*}
Although this is not the common notation used in the literature, we find it more convenient to distinguish the generators of $L_{K} (E)$ from the vertices and edges.
\end{remark}

\subsection{Abelian monoids associated to directed graphs}

Let $\textbf{CMon}_{0}$ be the category whose objects are abelian monoids and with morphisms that preserve the identity element. 

\begin{definition}
Let $R$ be a ring.  Let $\mathsf{M}_{\infty} ( R )$ be the ring $\bigcup_{ n = 1}^{ \infty } \mathsf{M}_{n} ( R )$ where we identify $\mathsf{M}_{n} (  R ) \subseteq \mathsf{M}_{n+1} ( R )$ by the homomorphism 
\begin{align*}
a \mapsto \begin{pmatrix} a & 0 \\ 0 & 0 \end{pmatrix} = a \oplus 0.
\end{align*}
Let $e, f \in \mathsf{M}_{\infty} (R)$ be idempotents.  We write $e \sim f$ if there exist $x , y \in \mathsf{M}_{\infty} (R)$ such that 
\begin{align*}
e = xy \quad \text{and} \quad y x = f.
\end{align*}
Define $\mathcal{V} [R]$ to be the monoid $\setof{ [ e ] }{ \text{$e$ an idempotent in $\mathsf{M}_{\infty} (R)$} }$ with addition defined as
\begin{align*}
[e] + [ f ] = \left[ e \oplus f \right].
\end{align*}
\end{definition}

\begin{definition}
Let $E = ( E^{0} , E^{1}, r_{E} , s_{E} )$ be a graph.  Set $\overline{M}_{E}$ to be the abelian monoid generated by 
\begin{align*}
\setof{ a_{v} }{v \in E^{0} } \cup \setof{ a_{v,S} }{ \text{$v$ an infinite emitter and $S$ a non-empty finite subset of $s_{E}^{-1}( v )$ } }
\end{align*}
with the relation given by 
\begin{align*}
a_{v} = \sum_{ e \in s_{E}^{-1} (  v  ) } a_{r_{E}(e)}
\end{align*} 
if $v$ is not a singular vertex and 
\begin{align*}
a_{v,S} + \sum_{ e \in S } b_{r_{E}(e)} = a_{v} \quad \text{and} \quad a_{v,S} + \sum_{ e \in S \setminus T }  a_{ r_{E} (e) }= a_{v,T} + \sum_{ e \in T\setminus S } a_{ r_{E} (e) }
\end{align*}
for all non-empty finite subsets, $S$ and $T$, of $s_{E}^{-1} (  v  )$ if $v$ is an infinite emitter.
\end{definition}

Ara, Moreno, and Pardo in \cite{amp:nonstablekthy} associated to every countable row-finite graph $E$ a monoid $M_{E}$ generated by $\setof{ a_{v} }{ v \in E^{0} }$ satisfying the relation
\begin{align*}
a_{v} = \sum_{ e \in s_{E}^{-1} (  v  ) } a_{ r_{E} ( e ) }
\end{align*}
for every $v \in E^{0}$ with $| s_{E}^{-1} (  v  ) | \neq 0$.  Moreover, they proved the following.

\begin{theorem}[Theorem~3.5 of \cite{amp:nonstablekthy}]\label{t:rowfinite}
Let $E$ be a countable row-finite graph, let $K$ be a field, and let 
\begin{align*}
\setof{ p_{v} , t_{e}, t_{e}^{*} }{ v \in E^{0} , e \in E^{1} }
\end{align*}
be a Leavitt $E$-family generating $L_{K} (E)$.  Then there exists a monoid isomorphism 
\begin{align*}
\ftn{ \gamma_{E} }{ M_{E} }{ \mathcal{V} [ L_{K} ( E ) ] }
\end{align*}
such that $\gamma_{E} ( [ a_{v}] ) = [ p_{v} ]$ for all $v \in E^{0}$.
\end{theorem}

\begin{remark}
It turns out that the techniques in \cite{amp:nonstablekthy} can be used for Leavitt path algebras $L_{K} (E)$ for an uncountable row-finite graph.   Since their results precede the notion of Leavitt path algebras for uncountable graphs, Ara, Moreno, and Pardo in \cite{amp:nonstablekthy} implicitly assumed that all graphs are countable.  Thus, we explicitly state the fact that $E$ is a countable row-finite graph here.
\end{remark}

The purpose of the paper is to generalize their results to arbitrary graphs.  We first prove that there is a monoid morphism $\gamma_{E}$ from $\overline{M}_{E}$ to $\mathcal{V} [ L_{K} (E) ]$ for arbitrary graphs.

\begin{lemma}\label{l:naturalhom}
Let $E = ( E^{0} ,  E^{1} , r_{E} , s_{E} )$ be a graph, let $K$ be a field, and let 
\begin{align*}
\setof{ p_{v} , t_{e}, t_{e}^{*} }{ v \in E^{0}, e \in E^{1} }
\end{align*}
be a Leavitt $E$-family generating $L_{K} ( E )$.  Then there exists a monoid morphism $\ftn{ \gamma_{E} }{ \overline{M}_{E} }{ \mathcal{V} [ L_{K} ( E ) ] }$ such that $\gamma_{E} ( [ a_{v} ] ) = [ p_{v} ]$ and $\gamma_{E} (  [ a_{v, S } ] ) = [ p_{v}  - \sum_{ e \in S} t_{e} t_{e}^{*}  ]$ when $v$ is an infinite emitter with $S$ a non-empty finite subset of $s^{-1}_{E} ( v  )$. 
\end{lemma}

\begin{proof}
Let $v \in E^0$ be a finite emitter. Note that
\begin{align*}
\gamma_E(a_v) &= [p_v] = \left[\sum_{e \in s_E^{-1}(v)} t_et_e^*\right] = \sum_{e \in s_E^{-1}(v)} [t_et_e^*] = \sum_{e \in s_E^{-1}(v)} [t_e^*t_e] = \sum_{e \in  s_E^{-1}(v)} [p_{r_E(e)}]\\
&= \sum_{e \in  s_E^{-1}(v)} \gamma_E(a_{r_E(e) }).
\end{align*}

Let $v \in E^0$ be an infinite emitter and $S$ be a non-empty finite subset of $s_E^{-1}(v)$.  Note that
\begin{align*}
\gamma_E(a_{v}) &= [p_v] = \left[p_v - \sum_{e \in S} t_et_e^* + \sum_{e \in S}t_et_e^* \right] = \left[p_v - \sum_{e \in S} t_et_e^*\right] + \left[\sum_{e \in S}t_et_e^* \right] \\
&= \left[p_v - \sum_{e \in S} t_et_e^*\right] + \sum_{e \in S} [t_et_e^*] = \left[p_v - \sum_{e \in S} t_et_e^*\right] + \sum_{e \in S} [t_e^*t_e]\\
&= \left[p_v - \sum_{e \in S} t_et_e^*\right] + \sum_{e \in S} [p_{r_E(e)} ] = \gamma_E(a_{v,S}) + \sum_{e \in S} \gamma_E(a_{r_E(e)}).
\end{align*}

Now let $T$ be a non-empty finite subset of $s_E^{-1}(v)$.  Then
\begin{align*}
\gamma_E(a_{v,S}) + \sum_{e \in S \setminus T} \gamma_E(a_{r_E(e)}) &= \left[p_v - \sum_{e \in S}t_et_e^*\right] + \sum_{e \in S \setminus T} [p_{r_E(e)}] \\
&=\left[p_v - \sum_{e \in S \cup T} t_et_e^* + \sum_{e \in T \setminus S} t_et_e^* \right] + \sum_{e \in S \setminus T} [t_e^*t_e]\\
&= \left[p_v - \sum_{e \in S \cup T} t_et_e^* \right] + \left[\sum_{e \in T \setminus S} t_et_e^* \right] + \sum_{e \in S \setminus T} [t_et_e^*]\\
&= \left[p_v - \sum_{e \in S \cup T} t_et_e^* \right] + \sum_{e \in T \setminus S} [t_et_e^* ] + \left[\sum_{e \in S \setminus T} t_et_e^*\right]\\
&= \left[p_v - \sum_{e \in S \cup T } t_et_e^* + \sum_{e \in S \setminus T} t_et_e^*\right] + \sum_{e \in T \setminus S} [t_e^*t_e ] \\
&= \left[p_v - \sum_{e \in T } t_et_e^*\right] + \sum_{e \in T \setminus S} [p_{r_E(e)}]\\
&=\gamma_E(a_{v,T}) + \sum_{e \in T \setminus S} \gamma_E(a_{r_E(e)}).
\end{align*}
We have just shown that $\gamma_{E}$ is a well-defined monoid homomorphism.
\end{proof}

\begin{remark}\label{r:monoids}
Note that if $E$ is a row-finite graph, then $\overline{M}_{E} = M_{E}$ and the natural monoid isomorphism in Theorem~\ref{t:rowfinite} is the same as the monoid morphism given in Lemma~\ref{l:naturalhom}.
\end{remark}

\section{Isomorphism for countable graphs}\label{countable}

In this section, we will show that the monoid morphism $\gamma_{E}$ given in Lemma~\ref{l:naturalhom} is a monoid isomorphism for countable graphs.  In order to do this, we will first show that there exists a monoid isomorphism from $\overline{M}_{E}$ to $\overline{M}_{F}$ which respects $\gamma_{E}$ and $\gamma_{F}$, where $F$ is the disingularization of $E$.

\begin{definition}[Definition~2.2 of \cite{ddmt:arbgraph}]
Let $E = ( E^{0} , E^{1} , r_{E} , s_{E} )$ be a countable graph.   The \emph{disingularization of $E$} is the graph $F$ defined as follows:
\begin{itemize}
\item[(1)] $F^0 = \{ w_0(v) \mid v \in E^0 \} \cup \{ w_n(v) \mid v \text{ is an infinite emitter or a sink and } n \in \N\}$. 

\item[(2)] $F^1$ is the union of $\{ g_n^v \mid n \in \N_{ \geq 0 } , v \in E^0_{\mathrm{sing}} \}$ and  
\begin{align*}
\{f_n^v \mid v \in E^0, v \text{ is not a sink, and } 0 \leq n < |s_E^{-1}(v)| \}.
\end{align*}

\item[(3)] The range and source maps $r_F$ and $s_F$ is defined by:  
\begin{itemize}
\item[(a)] If $s_E^{-1}(v) = \{e_k^v \mid  0 \leq k < |s_E^{-1}(v)| \}$, then 
\begin{align*}
s_{F} ( f_{n}^{v} ) = 
\begin{cases}
w_{0} (v), &\text{if $v$ is a finite emitter} \\
w_{n} (v), &\text{if $v$ is an infinite emitter}.
\end{cases}
\end{align*}
and 
\begin{align*}
r_{F} ( f_{n}^{v} ) = w_{0} ( r_{E} ( e_{n}^{v} ) ).
\end{align*}

\item[(b)] $r_{F} ( g_{n}^{v} ) = w_{n+1} ( v )$ and $s_{F} ( g_{n}^{v} ) = w_{n} (v)$.
\end{itemize}
 \end{itemize}
 \end{definition}

\begin{proposition}\label{p:monoid-dising-iso}
Let $E$ be a countable graph. Let $F$ be the disingularization of $E$.  Then there exists a monoid isomorphism $\ftn{ \varphi_{E} }{ \overline{M}_E }{ \overline{M}_F }$.
\end{proposition}

\begin{proof}

Let $v \in E^0$ and $n \in \N$. Let $T_{n}^{v} = \{ e_{k}^{v} \mid 0 \leq k \leq n ,\  s_{E} ( e_{k}^{v} ) = v\} $.  Define 
$$  
\lambda : \cup_{v \text{ is inf. emit.}} \cup_{n \in \N} T_{n}^{v} \rightarrow F^{0}
$$
by $\lambda(e_{n}^{v}) = f_{n}^{v}$.  Consider the following map: $\varphi: \overline{M}_E \rightarrow \overline{M}_F$ where
$$
\varphi(a_v) = b_{w_0(v)} \quad \text{ and } \quad \varphi(a_{v,S}) = b_{w_{n+1}(v)} + \sum_{f \in \lambda (T_{n}^{v}) \backslash \lambda (S)} b_{r_F(f)}. 
$$
Here $n$ is the largest number such that  $e_{n}^{v} \in S$.  We want to show that $\varphi$ respects our relations from $\overline{M}_E$. Let $v$ be a finite emitter in $E$ with $n$ edges coming out of it. Note that 
$$a_v = \sum_{e \in s_{E}^{-1}(v)} a_{r_{E}(e)}. $$
We want to show that 
$$
\varphi(a_v) =  \sum_{e \in s_{E}^{-1}(v)} \varphi (a_{r_{E}(e)}). 
$$
Note that 
\begin{align*} \varphi (a_v) = b_{w_0(v)} &= \sum_{f \in s_{F}^{-1}(w_0(v))} b_{r_{F}(f)} = \sum_{i=0}^{n-1} b_{r_{F}(f_{i}^{v})} = \sum_{i=0}^{n-1} b_{w_0(r_E(e_{i}^{v}))} = \sum_{i=0}^{n-1} \varphi (a_{r_E(e_{i}^{v})}) \\ 
&= \sum_{e \in s_{E}^{-1}(v)} \varphi (a_{r_{E}(e)}).
 \end{align*}
Let $v$ be an infinite emitter in $E$. Let $S$ be a non-empty finite subset of  $ s_{E}^{-1}(v)$. Note that $a_v = a_{v,S} + \sum_{e \in S} a_{r_{E}(e)}$.  We want to show that 
$$
\varphi (a_v) = \varphi (a_{v,S}) + \sum_{e \in S} \varphi (a_{r_{E}(e)}).
$$

Let $n$ be the greatest number such that $e_{n}^{v} \in S$.  Note that $b_{r_{F}(g_{k-1}^{v})} = b_{r_{F}(g_{k}^{v})} + b_{r_{F}(f_{k}^{v})}$ for all $k > 0$. By repeated use of the recursive formula, 
\begin{align*}
\varphi(a_v) = b_{w_0(v)} &= \sum_{f \in s_{F}^{-1}(w_0(v))} b_{r_{F}(f)} = b_{r_{F}(g_{0}^{v})} + b_{r_{F}(f_{0}^{v})} = b_{r_{F}(g_{n}^{v})} + \sum_{i=0}^{n} b_{r_{F}(f_{i}^{v})} \\
&= b_{w_{n+1}(v)} + \sum_{f \in \lambda(T_{n}^{v})} b_{r_{F}(f)}= b_{w_{n+1}(v)} + \sum_{f \in \lambda(T_{n}^{v}) \backslash \lambda(S)} b_{r_{F}(f)} + \sum_{f \in \lambda(S)} b_{r_{F}(f)} \\
&= \varphi (a_{v,S}) + \sum_{f \in \lambda(S)} b_{r_{F}(f)}= \varphi (a_{v,S}) + \sum_{e \in S} b_{w_{0}(r_{E}(e))} \\
&= \varphi (a_{v,S}) + \sum_{e \in S} \varphi (a_{r_{E}(e)}). 
\end{align*}

Let $v$ be an infinite emitter in $E$. Let $S$ and $T$ be non-empty finite subsets of  $ s_{E}^{-1}(v)$. Note that
$$a_{v,S} + \sum_{e \in S \backslash T} a_{r_{E}(e)} = a_{v,T} + \sum_{e \in T \backslash S} a_{r_{E}(e)}.$$
Let $n_S$ be the greatest number such that $e_{n_{S}}^{v} \in S$ and $n_T$ be the greatest number such that $e_{n_{T}}^{v} \in T$. Without loss of generality we can assume that $n_S \leq n_T$.  Thus, we want to show that
$$
\varphi (a_{v,S}) + \sum_{e \in S \backslash T} \varphi (a_{r_{E}(e)}) = \varphi (a_{v,T}) + \sum_{e \in T \backslash S} \varphi (a_{r_{E}(e)}).
$$
Note that
\begin{align*} \varphi (a_{v,S}) + \sum_{e \in S \backslash T} \varphi (a_{r_{E}(e)}) &= b_{w_{n_{S}+1}(v)} + \sum_{f \in \lambda(T_{n_{S}}^{v}) \backslash \lambda (S)} b_{r_{F}(f)} + \sum_{e \in S \backslash T} b_{w_{0}(r_{E}(e))} \\
&= b_{w_{n_{T}+1}(v)} + \sum_{i=n_{S}+1}^ {n_{T}} b_{r_{F}(f_{i}^{v})}  \\
&\qquad +\sum_{f \in \lambda(T_{n_{S}}^{v}) \backslash \lambda (S)} b_{r_{F}(f)} + \sum_{e \in S \backslash T} b_{w_{0}(r_{E}(e))}\\
 &= b_{w_{n_{T}+1}(v)} + \sum_{f \in \lambda(T_{n_{T}} ^{v}) \backslash \lambda (T_{n_{S}}^{v})} b_{r_{F}(f)}    \\
 &\qquad +           \sum_{f \in \lambda(T_{n_{S}}^{v}) \backslash \lambda (S)} b_{r_{F}(f)} + \sum_{f \in \lambda (S) \backslash \lambda (T)} b_{r_{F}(f)}\\
 &= b_{w_{n_{T}+1}(v)} + \sum_{f \in \lambda(T_{n_{T}} ^{v}) \backslash \lambda (S)} b_{r_{F}(f)}    + \sum_{f \in \lambda (S) \backslash \lambda (T)} b_{r_{F}(f)}\\
 &=  b_{w_{n_{T}+1}(v)} + \sum_{f \in \lambda(T_{n_{T}} ^{v}) \backslash \lambda (T)} b_{r_{F}(f)}    + \sum_{f \in \lambda (T) \backslash \lambda (S)} b_{r_{F}(f)}\\
 &=\varphi (a_{v,T}) + \sum_{e \in T \backslash S} \varphi (a_{r_{E}(e)}).
 \end{align*}
 Thus we have shown that $\varphi$ respects the given relations, which implies that $\varphi$ is a well-defined monoid morphism.

We now construct the inverse of $\varphi$.  Define $\psi: \overline{M}_F \to \overline{M}_E$ by
\begin{align*}
\psi ( b_{ w_{0} (v) } ) = a_{v}
\end{align*}
for all $v \in E^{0}$ and
\begin{align*}
\psi( b_{ w_{n} (v) } ) = 
\begin{cases}
a_{v, T_{n-1}^{v} }, &\text{if $v$ is an infinite emitter} \\
a_{v},			&\text{if $v$ is a sink}.
\end{cases}
\end{align*}
We want to show that $\psi$ respects the relations of $\overline{M}_F$, i.e., 
$$
\displaystyle b_{w_0(v)} = \sum_{s_F^{-1}(w_0(v))} b_{r_F(f)}.
$$
Hence, if $v$ is a finite emitter and not a sink where $\psi(b_{w_0(v)}) = a_v$, then  we must prove that 
$$
\psi(b_{w_0(v)})  = \sum_{f \in s_F^{-1}(w_0(v))} \psi(b_{r_F(f)}).
$$
Let $v$ be a finite emitter that is not a sink.  Note that 
$$
\psi(b_{w_0(v)}) = a_v = \sum_{e \in s_E^{-1}(v))} a_{r_E (e)} = \sum_{e \in s_E^{-1}(v)} \psi(b_{w_0(r_E(e))}) = \sum_{f \in s_F^{-1}(w_0(v))} \psi(b_F(f)).
$$
Let $v$ be an infinite emitter. Here we have two cases

\begin{itemize}
\item[(1)] For $w_0(v)$, $\psi(b_{w_0(v)}) = a_v$. Let $n \in \mathbb Z_{\geq 0}$.  
\begin{align*}
\psi ( b_{r_{F} ( g_{0}^{v} ) } ) + \psi ( b_{r_{F} ( f_{0}^{v} ) } ) &= \psi ( b_{ w_{1} ( v ) } ) + \psi ( b_{ r_{F} ( f_{0}^{v} ) } ) = a_{v, T_{0}^{v}} + \psi( b_{ w_{0} ( r_{E} ( e_{0}^{v} ) ) } ) \\
	&= a_{ v, T_{0}^{v} } + a_{ r_{E} ( e_{0}^{v} ) } = a_{v} = \psi ( b_{ w_{0} (v) } ).
\end{align*}

\item[2)] For $w_n(v)$ with $n \geq 1$, $\psi(b_{w_n(v)}) = a_{v, T_{n-1}^v}$. We want to show that $\psi$ respects the relation given by 
$$
\displaystyle b_{w_n(v)} = b_{w_{n+1}(v)} + b_{r_F (f_n^v)} = \sum_{f \in s_F^{-1}(w_n(v))} b_{r_F(f)}.
$$   
Let $n \geq 1$.  Note that 
\begin{align*}
\displaystyle \psi(b_{w_n(v)}) = a_{v, T_{n-1}^v} &= a_{v, T_n^v} + a_{r_E(e_{n}^v)} = \psi(b_{w_{n + 1}(v)}) + \psi(b_{w_0(r_E (e_{n}^v))}) \\
&= \psi(b_{w_{n+1}(v)}) + \psi( b_{r_F (f_n^v)}) = \sum_{f \in s_F^{-1}(w_n(v))} \psi(b_{r_F(f)}).
\end{align*}
\end{itemize}
Now let $v$ be a sink and $n \geq 0$. Note that
\begin{align*}
\psi(b_{w_n(v)} ) = a_v = \psi(b_{w_{n+1} (v)}).
\end{align*}
Thus we have shown that $\psi$ respects the given relations, which implies that $\psi$ is a well-defined monoid morphism.

We will now show that $\psi$ is indeed the inverse of $\varphi$.  Let $v \in E^0$.  Then
$$
\psi(\varphi(a_v)) = \psi(b_{w_0(v)}) = a_v.
$$
Consider an infinite emitter $v \in E^{0}$ and let $S$ be a non-empty finite subset of $s_{E}^{-1} (v)$. Let $n$ be the largest integer such that $e_{n}^{v} \in S$.  Note that 
\begin{align*} 
\psi(\varphi(a_{v, S})) &= \psi\left(b_{w_{n+1}(v)}+ \sum_{f \in \lambda(T_n^v)\setminus \lambda(S)} b_{r_F(f)} \right) = \psi(b_{w_{n+1}(v) }) + \sum_{f \in \lambda(T_n^v) \setminus \lambda (S)} \psi (b_{r_F (f)})\\
& = a_{v, T_n^v} + \sum_{f \in \lambda(T_n^v)\setminus \lambda (S)} \psi(b_{r_F(f)}) = a_{v, T_n^v} + \sum_{e \in T_n^v \setminus S} a_{r_E(e)} = a_{v, S}.
\end{align*}
Since $\psi \circ \varphi$ is the identity function on the generators of $\overline{M}_{E}$, we have that $\psi \circ \varphi = \id_{ \overline{M}_{E} }$.

We now show that $\varphi \circ \psi = \id_{ \overline{M}_{F} }$.  Consider $w_0(v) \in F^0$.  Note that $\varphi(\psi(b_{w_0(v)})) = \varphi(a_v) = b_{w_0(v)}$.  Consider $w_n(v) \in F^0$ where $v \in E^0$ is a sink.  Note that 
\begin{align*}
\varphi (\psi(b_{w_n(v)})) = \varphi(a_v) = b_{w_0(v)} = b_{w_n(v)}.
\end{align*}

Consider $w_n(v) \in F^0$ where $v \in E^0$ is an infinite emitter.  Note that 
\begin{align*}
\displaystyle \varphi(\psi(b_{w_n(v)})) = \varphi(a_{v, T_{n-1}^v}) = b_{w_n(v)} + \sum_{f \in \lambda (T_{n-1}^v) \setminus \lambda(T_{n-1}^v)} b_{r_F(f)} = b_{w_n(v)}.
\end{align*}
Since $\varphi \circ \psi$ is the identity function on the generators of $\overline{M}_{F}$, we have that $\varphi \circ \psi = \id_{ \overline{M}_{F} }$. 

We have just shown that $\varphi$ and $\psi$ are inverse functions.  Hence, $\varphi$ is a monoid isomorphism which implies that $\overline{M}_E \cong \overline{M}_F$.
\end{proof}

The next lemma is probably well-known to the experts in the field but we were not able to find a reference.  For the convenience of the reader we provide the proof here. 

\begin{lemma}\label{l:corner}
Let $R$ be a ring and let $e$ be an idempotent in $R$.  If $p, q$ are idempotents in $e R e$ and there exist $x, y \in R$ such that $p = xy$ and $yx = q$, then there exist $v,w \in e R e$ such that $p = vw$ and $wv = q$.  Consequently, the usual embedding $\ftn{ \iota }{ e R e }{ R }$ induces an injective monoid morphism $\ftn{ \mathcal{V} [\iota ] }{ \mathcal{V} [ e R e ] } { \mathcal{V} [ R ] }$.
\end{lemma}

\begin{proof}
Suppose $p, q$ are idempotents in $e R e$ and there exist $x, y \in R$ such that $p = xy$ and $yx = q$.  Set $v = p x q$ and $w = q y p$.  Then $vw = p x q y p = p x y x y p = p p p p = p$ and $wv = qypxq = q y x y x q = q q q q = q$.  
\end{proof}

\begin{theorem}\label{t:iso-countable}
Let $E = ( E^{0} , E^{1}, r_{E} , s_{E} )$ be a countable graph, let $F = ( F^{0} , F^{1}, r_{F}, s_{F} )$ be the desingularization of $E$, and let $K$ be a field.  Let $\varphi_{E}$ be given in Proposition~\ref{p:monoid-dising-iso}.  Then there exists a homomorphism $\ftn{ \kappa_{E} }{ L_{K} ( E ) }{ L_{K} ( F ) }$ such that $\mathcal{V} [ \kappa_{E} ]$ is a monoid isomorphism and the diagram
\begin{align*}
\xymatrix{
\overline{M}_{E} \ar[r]^-{ \gamma_{E} } \ar[d]_{ \varphi_{E}  } & \mathcal{V} [ L_{K} ( E ) ] \ar[d]^{ \mathcal{V} [\kappa_{E} ] } \\
\overline{M}_{F} \ar[r]_-{ \gamma_{F} } & \mathcal{V} [ L_{K} ( F ) ]
}
\end{align*} 
is commutative.  Consequently, $\gamma_{E}$ is a monoid isomorphsm.
\end{theorem}

\begin{proof}
Let $\setof{ p_{v}, t_{f}, t_{f}^{*} }{ v \in F^{0}, f \in F^{1} }$ be a set of Leavitt $F$-family generating $L_{K} ( F )$ and let $\setof{ P_{v}, T_{e}, T_{e}^{*} }{ v \in E^{0}, e \in E^{1} }$ be a set of Leavitt $E$-family generating $L_{K} ( E )$.  Let $e \in E^{1}$.  Then $e = e_{j}^{v}$ where $v = s_{E} (e)$.  If $v$ is not a singular vertex, then set $s_{e} = t_{f_{j}^{v}}$ and $s_{e}^{*} = t_{f_{j}^{v}}^{*}$.  Suppose $v$ is an infinite emitter.  Set $s_{e} = t_{\alpha_{j}} = t_{g_{0}^{v}} \cdots t_{g_{j-1}^{v}} t_{ f_{j}^{v}}$ and $s_{e}^{*} = t_{ \alpha_{j}}^{*} = t_{ f_{j}^{v} }^{*} t_{g_{j-1}^{v}}^{*} \cdots t_{g_{0}^{v}}^{*}$, where $\alpha_{j} = g_{0}^{v} g_{1}^{v} \dots g_{j-1}^{v} f_{j}^{v}$.  By Proposition~5.5 of \cite{gaga:leavittpath2}, there exists a monomorphism $\ftn{ \kappa_{E} }{ L_{K} ( E ) }{ L_{K} ( F ) }$ such that $\kappa_{ E } ( P_{v} ) = p_{w_0(v)}$ for each $v \in E^{0}$, and $\kappa_{E} ( T_{e} ) = s_{e}$ and $\kappa_{E} ( T_{e}^{*} ) = s_{e}^{*}$ for each $e \in E^{1}$.   

We now show that the diagram is commutative.  Let $v \in E^0$. Then 
\begin{align*}
\mathcal{V} [ \kappa_E ] (\gamma_E(a_v)) = \mathcal{V} [ \kappa_E ] ([P_v]) = [p_{w_0(v)}]  \quad \text{ and } \quad \gamma_F(\varphi_E(a_v)) = \gamma_F(b_{w_0(v)}) = [p_{w_0(v)}].
\end{align*}
Let $v \in E^0$ be an infinite emitter, $S$ be a non-empty finite subset of $s_E^{-1}(v)$, and $n = \max \{ k \in \Z_{ \geq 0 } : e_k^v \in S \}$. Then
\begin{align*}
&\mathcal{V} [ \kappa_E ] (\gamma_E(a_{v,S})) \\
&\qquad = \mathcal{V} [ \kappa_E ] \left(\left[P_v - \sum_{e \in S} T_eT_e^*\right]\right) \quad = \mathcal{V} [ \kappa_E ] \left( \left[P_v - \sum_{e \in T_n^v}T_eT_e^* + \sum_{e \in T_n^v \setminus S} T_e T_e^* \right] \right)\\
&\qquad = \mathcal{V} [ \kappa_E ] \left( \left[ P_v - \sum_{e \in T_n^v}T_eT_e^* \right] \right) + \mathcal{V} [\kappa_E ] \left( \left[\sum_{e \in T_n^v \setminus S}T_eT_e^* \right] \right) \\
&\qquad = \mathcal{V} [ \kappa_E ] \left( \left[ P_v - \sum_{e \in T_n^v}T_eT_e^* \right] \right) + \mathcal{V} [\kappa_E ] \left( \sum_{e \in T_n^v \setminus S} \left[ T_eT_e^* \right] \right) \\
&\qquad = \mathcal{V} [ \kappa_E ] \left( \left[ P_v - \sum_{e \in T_n^v}T_eT_e^* \right] \right) + \sum_{e \in T_n^v \setminus S} \mathcal{V} [ \kappa_E ] \left( \left[ T_eT_e^* \right] \right). 
\end{align*}
Therefore,
\begin{align*}
&\mathcal{V} [ \kappa_E ] \left( \left[ P_v - \sum_{e \in T_n^v}T_eT_e^* \right] \right) + \sum_{e \in T_n^v \setminus S} \mathcal{V} [ \kappa_E ] \left( \left[ T_eT_e^* \right] \right)  \\
&\qquad =  \left[ p_{w_0(v)} - \sum_{e \in T_n^v}s_e s_e^* \right] + \sum_{ e \in T_{n}^{v} \setminus S} [ s_{e} s_{e}^{*} ] \\
&\qquad = \left[ p_{w_{0} (v) } - t_{ f_{0}^{v} } t_{f_{0}^{v}}^{*} - \sum_{i = 1}^{n} t_{g_{0}^{v} } t_{g_{1}^{v} }  \cdots t_{g_{i-1}^{v} } t_{f_{i}^{v}} t_{f_{i}^{v}}^{*}  t_{g_{i-1}^{v} }^{*} \cdots t_{g_{0}^{v} }^{*} \right] + \sum_{e \in T_n^v \setminus S}  [s_e s_e^*]   \\
&\qquad = \left[ p_{w_{0} (v) } - t_{ f_{0}^{v} } t_{f_{0}^{v}}^{*} - \sum_{i = 1}^{n} t_{g_{0}^{v} } t_{g_{1}^{v} }  \cdots t_{g_{i-1}^{v} } ( t_{ w_{i}(v)} - t_{g_{i}^{v}} t_{g_{i}^{v}}^{*} ) t_{g_{i-1}^{v} }^{*} \cdots t_{g_{0}^{v} }^{*} \right] + \sum_{e \in T_n^v \setminus S}  [s_e s_e^*] \\
&\qquad = \left[ p_{ w_{0} (v) } - t_{ f_{0}^{v} } t_{f_{0}^{v}}^{*}  - t_{ g_{0}^{v} } t_{g_{0}^{v}}^{*} + t_{g_0^v} \cdots t_{g_{n-1}^v} t_{g_n^v} t_{g_n^v}^* t_{g_{n-1}^v}^* \cdots t_{g_0^v}^*\right] + \sum_{e \in T_n^v \setminus S}  [s_e s_e^*]  \\
&\qquad = \left[  t_{g_0^v} \cdots t_{g_{n-1}^v} t_{g_n^v} t_{g_n^v}^* t_{g_{n-1}^v}^* \cdots t_{g_0^v}^*\right] + \sum_{e \in T_n^v \setminus S}  [s_e s_e^*].  
\end{align*}
Note that $[  t_{g_n^v}^* t_{g_{n-1}^v}^* \cdots t_{g_0^v}^*t_{g_0^v} \cdots t_{g_{n-1}^v} t_{g_n^v} ] = [ t_{ g_{n}^{v} }^{*} t_{ g_{n}^{v} } ]$.  Since 
\begin{align*}
[ p_{ w_{n+1} (v) } ] + \sum_{f \in \lambda(T_n^v) \setminus \lambda(S)} [p_{r_F(f)}] = [ p_{ w_{n+1} (v) } ] + \sum_{ e_{k}^{v} \in T_{n}^{v} \setminus S } [ p_{ r_{F} ( f_{k}^{v} ) }] = [ t_{ g_{n}^{v} }^{*} t_{ g_{n}^{v} } ] +  \sum_{ e \in T_{n}^{v} \setminus S } [ s_{e}^{*} s_{e}  ],  
\end{align*}
we have that 
\begin{align*}
\mathcal{V} [ \kappa_E ] (\gamma_E(a_v)) = [ p_{ w_{n+1} (v) } ] + \sum_{f \in \lambda(T_n^v) \setminus \lambda(S)} [p_{r_F(f)}]. 
\end{align*}

Since
\begin{align*}
&\gamma_F(\varphi_E(a_{v,S})) \\
&\qquad = \gamma_F\left(b_{w_{n+1}(v)} + \sum_{f \in \lambda(T_n^v) \setminus \lambda(S)} b_{r_F(f)}\right) = \gamma_F (b_{w_{n+1}(v)}) + \gamma_F \left( \sum_{f \in \lambda(T_n^v) \setminus \lambda(S)} b_{r_F(f)} \right) \\
&\qquad = \gamma_F (b_{w_{n+1}(v)}) + \sum_{f \in \lambda(T_n^v) \setminus \lambda(S)} \gamma_F \left(  b_{r_F(f)} \right)= [p_{w_{n+1}(v)}] + \sum_{f \in \lambda(T_n^v) \setminus \lambda(S)} [p_{r_F(f)}],
\end{align*}
we have that 
\begin{align*}
\mathcal{V} [ \kappa_E ] (\gamma_E(a_{v,S})) = \gamma_F(\varphi_E(a_{v,S})).
\end{align*}
Since $\mathcal{V} [ \kappa_{E} ] \circ \gamma_{E}$ is equal to $\gamma_{F} \circ \varphi_{E}$ on the generators of $\overline{M}_{E}$, we have that $\mathcal{V} [ \kappa_{E} ] \circ \gamma_{E} = \gamma_{F} \circ \varphi_{E}$.

We now show that $\mathcal{V} [ \kappa_{E} ]$ is a monoid isomorphism.  First note that since $F$ is a row-finite graph, by Theorem~\ref{t:rowfinite} and Remark~\ref{r:monoids}, $\gamma_{F}$ is a monoid isomorphism.  By Proposition~2.1, $\varphi_{E}$ is a monoid isomorphism.  By the commutativity of the diagram, we have that $\mathcal{V} [ \kappa_{E} ]$ is surjective.

Set $E^{0} = \setof{ v_{n} }{ n \in \N }$.  Set $P_{n,E} = \sum_{ k = 1}^{n} P_{ v_{k} }$ and set $P_{n, F } = \sum_{ k = 1}^{n} p_{v_{k} }$.  By the proof of Theorem~5.6 of \cite{gaga:leavittpath2}, 
\begin{align*}
\ftn{ \kappa_{n} = ( \kappa_{E} ) \vert_{ P_{n,E} L_{K} ( E ) P_{n, E } } }{ P_{n,E} L_{K} ( E ) P_{n, E } }{ P_{n,F} L_{K} ( F ) P_{n, F } }
\end{align*}
is an isomorphism and the diagram
\begin{align*}
\xymatrix{
P_{n,E} L_{K} ( E ) P_{n, E } \ar[r]^-{\iota_{n, E}} \ar[d]_{ \kappa_{n} } & L_{K} ( E ) \ar[d]^{ \kappa_{E} } \\ 
P_{n, F} L_{K} ( F ) P_{n, F } \ar[r]_-{ \iota_{n, F} } 			&  L_{K} ( F ) 
}
\end{align*}
is commutative.  Note that $\bigcup_{ n = 1}^{ \infty } P_{n,E} L_{K} ( E ) P_{n, E } = L_{K} ( E )$.  Suppose $e$ and $q$ are idempotents in $\mathsf{M}_{m} ( L_{K} (E ) )$ such that $\mathcal{V} [ \kappa_{E} ] ( [ e ] ) = \mathcal{V} [ \kappa_{E} ]( [ q ] )$.  Then there exists $n \in \N$ such that $e$ and $q$ are idempotents in $\mathsf{M}_{m} ( P_{n,E} L_{K} ( E ) P_{n, E } )$.  By the commutativity of the above diagram, 
\begin{align*}
\mathcal{V} [ \iota_{n,F} ] \circ \mathcal{V} [ \kappa_{n} ] ( [ e ] ) = \mathcal{V} [ \iota_{n,F} ] \circ \mathcal{V} [ \kappa_{n} ] ( [ q ] ).
\end{align*}  
By Lemma~\ref{l:corner}, $\mathcal{V} [ \kappa_{n} ] ( [ e ] )  = \mathcal{V} [ \kappa_{n} ] ( [ q ] )$.  Since $\kappa_{n}$ is an isomorphism, we have that $[ e ] = [ q ]$.  Therefore, $\mathcal{V} [ \kappa_{E} ]$ is injective.  Hence, $\mathcal{V} [ \kappa_{E} ]$ is a monoid isomorphism.
\end{proof}

\begin{corollary}
Let $E = ( E^{0} , E^{1} , r_{E} , s_{E} )$ be a countable graph and let $\ftn{ \iota_{E} }{ L_{\C} ( E ) }{ C^{*} ( E ) }$ be the natural inclusion.  Then $\mathcal{V} [ \iota_{E} ]$ is a monoid isomorphism.
\end{corollary}

\begin{proof}
Let $F$ be the desingularization of $E$.  By the proof of Lemma~2.9 of \cite{ddmt:arbgraph} and the definition of $\kappa_{E}$, there exists a $*$-homomorphism of $\ftn{ \overline{\kappa}_{E} }{C^{*} (E) }{ C^{*} ( F ) }$ such that  
\begin{align*}
\xymatrix{
L_{\C} ( E ) \ar[d]_{ \kappa_{E} } \ar[r]^-{ \iota_{E} } & C^{*} ( E ) \ar[d]^{ \overline{\kappa}_{E} } \\
L_{\C} ( F ) \ar[r]_-{\iota_{F}} & C^{*} ( F )
}
\end{align*}
is commutative.  Applying the functor $\mathcal{V} [ - ]$, we get that the diagram
\begin{align*}
\xymatrix{
\mathcal{V} [ L_{\C} ( E ) ]  \ar[d]_{ \mathcal{V} [ \kappa_{E} ] } \ar[r]^-{ \mathcal{V} [ \iota_{E} ] } & \mathcal{V} [ C^{*} ( E ) ] \ar[d]^{ \mathcal{V} [ \overline{\kappa}_{E}]  } \\
\mathcal{V} [ L_{\C} ( F ) ] \ar[r]_-{ \mathcal{V} [ \iota_{F} ] } & \mathcal{V} [ C^{*} ( F ) ]
}
\end{align*}
commutes.

 Let $\setof{ P_{v,F} , T_{e, F} }{ v \in F^{0}, e \in F^{1} }$ be a universal set of Cuntz-Kreiger $F$-family generating $C^{*} (F)$.  Then by Theorem~2.11 of \cite{ddmt:arbgraph}, there exists a projection $P$ in the multiplier algebra of $C^{*} (F)$ such that $\overline{\kappa}_{E} ( C^{*} (E) ) = P C^{*} ( F ) P$ and $P C^{*} (F) P$ is not contained in a closed ideal of $C^{*} (F)$, where $P = \sum_{ e \in E^{0} } P_{ v, F }$ (the sum converges in the strict topology).  

By Theorem~7.1 of \cite{amp:nonstablekthy}, $\mathcal{V} [ \iota_{F} ]$ is a monoid isomorphism.  By Theorem~\ref{t:iso-countable}, $\mathcal{V} [\kappa_{E} ]$ is a monoid isomorphism.  Hence, $\mathcal{V} [ \overline{\kappa}_{E} ]$ is surjective.  Suppose $e$ and $f$ are idempotents in $\mathsf{M}_{\infty} ( P C^{*}(F) P )$ such that $\mathcal{V} [ \overline{\kappa}_{E} ] ( [ e ] )  = \mathcal{V} [ \overline{\kappa}_{E} ]( [ f ] )$.  Note that
\begin{align*}
P C^{*} ( F ) P = \overline{ \bigcup_{ n = 1}^{ \infty } P_{n} C^{*} (F) P_{n} }
\end{align*}
where $P_{n} = \sum_{ k = 1}^{n} P_{v_{k} , F }$ with $E^{0} = \setof{ v_{k} }{ k \in \N }$.  By Proposition~4.5.2 of \cite{blackadarB}, every idempotent in $\mathsf{M}_{\infty} ( P C^{*} ( F ) P )$ is equivalent to an idempotent in $\mathsf{M}_{\infty} ( P_{n} C^{*} (F) P_{n} )$.  Hence, there exist $e' , f' \in  \mathsf{M}_{\infty} ( P_{n} C^{*} ( F ) P_{n} )$ for some $n \in \N$, and $[ \overline{\kappa}_{E} ( e )  ] = [ e ' ]$ and $[ \overline{\kappa}_{E} ( f )  ] = [ f' ]$ in $\mathcal{V} [ P_{n} C^{*} ( F ) P_{n} ]$.  Since $\mathcal{V} [ \overline{\kappa}_{E} ] ( [ e ] )  = \mathcal{V} [ \overline{\kappa}_{E} ]( [ f ] )$ in $\mathcal{V} [ C^{*} ( F ) ]$, we have that $[ e' ] = [ f' ]$ in $\mathcal{V} [ C^{*} ( F ) ]$.  By Lemma~\ref{l:corner}, $[ e' ] = [ f' ]$ in $\mathcal{V} [ P_{n} C^{*} (F) P_{n} ]$.  Therefore, $[ \overline{\kappa}_{E} ( e )  ] = [ \overline{\kappa}_{E} ( f )  ]$ in $\mathcal{V} [ P C^{*} ( F ) P ]$ which implies that $[e] = [f]$ in $\mathcal{V} [ C^{*} ( E ) ]$ since $\ftn{ \overline{\kappa}_{E} }{ C^{*} ( E ) }{ P C^{*}(F) P }$ is a $*$-isomorphism.  Hence, $\mathcal{V} [ \overline{\kappa}_{E} ]$ is injective.  Thus, $\mathcal{V} [ \overline{\kappa}_{E} ]$  is a monoid isomorphism.
 
By the above paragraphs, 
\begin{align*}
\mathcal{V} [ \iota_{E} ] = \mathcal{V} [ \overline{\kappa}_{E} ]^{-1} \circ \mathcal{V} [ \iota_{F} ]  \circ \mathcal{V} [ \kappa_{E} ]
\end{align*} 
is a monoid isomorphism.
\end{proof}

\section{Isomorphism for arbitrary directed graphs}\label{uncountable}

To prove that $\ftn{ \gamma_{E}  }{ \overline{M}_{E} }{ \mathcal{V} [ L_{K} (E) ] }$ is a monoid isomorphism for an arbitrary graph $E$, we will use the fact that $\gamma_{F}$ is a monoid isomorphism for countable graphs and the fact that $E$ can be expressed as a direct limit of countable graphs.  In order to use these facts, we will need to prove that $E \mapsto \overline{M}_{E}$ is a continuous functor for direct limits with morphisms being CK-morphisms as defined in \cite{kg:leavittlimits}, p. 8.

We begin by establishing some results in the category $\textbf{CMon}_{0}$ which are probably well-known to the experts, but for which we were unable to find a reference.  For the convenience of the reader, we provide the proof here.  The first fact is that direct limits exist in this category.


\begin{lemma}\label{l:direct-limits-monoids}
Let $\mathcal{S} = \left( ( S_{i} )_{ i \in I } , ( \mu_{ij} )_{ i \leq j , \text{ in } I} \right)$ be a direct system in $\textbf{CMon}_{0}$.  Then there exists an abelian monoid $S_{\infty}$ together with morphisms $\ftn{ \mu_{i, \infty} }{ S_{i} }{ S_{\infty} }$ such that 
\begin{itemize}
\item[(a)] $\mu_{i,\infty}  = \mu_{j, \infty }  \circ \mu_{ij}$ for all $i \leq j$ in $I$;

\item[(b)] for each $a \in S_{\infty}$, there exists $a_{1} \in S_{i}$ for some $i \in I$ with $a = \mu_{i, \infty}  ( a_{1} )$; and

\item[(c)] given any other abelian monoid $M$ and morphisms 
\begin{align*}
\left( \ftn{ \psi_{i} }{ S_{i} }{ M } \right)_{ i \in I }
\end{align*} 
such that $\psi_{i} = \psi_{j} \circ \mu_{ij}$ for all $i \leq j$ in $I$, there exists a morphism $\ftn{ \psi }{ S_{\infty} }{ M }$ such that $\psi_{i} = \psi \circ \mu_{i, \infty}$ for all $i \in I$.
\end{itemize}
\end{lemma}

 \begin{proof}
 Let $S$ denote the disjoint union of the sets $S_i$, for $i \in I$.
   Define a relation on $S$ as follows.  If $s \in S_i$ and $t \in S_j$ then $s \sim t$ if and only if there exists $k \in I$ such that $k \geq i, j$ in $I$ and $\mu_{ik}(s)=\mu_{jk}(t)$.  One easily shows that $\sim$ is an equivalence relation.  Define $S_{\infty}=S/\sim$ and let $[s]$ be the equivalence class corresponding to $s \in S_i$.  Given $s \in S_i$ and $t \in S_j$, define $[s]+[t]=[\mu_{ik}(s)+\mu_{jk}(t)]$ for any $k \geq i, j$ in $I$.  An easy check shows that this operation is well-defined, and commutative with identity element $[0]$.
 
Now define $\mu_{i, \infty }: S_i \rightarrow S_{\infty}$ by $\mu_{i,\infty} (s)=[s]$.  Then it follows easily that $\mu_{i, \infty}$ is an identity-preserving monoid morphism with the property that $\mu_{i, \infty }=\mu_{j, \infty} \circ \mu_{ij}$ for all $i \leq j$ in $I$.  Also, (b) follows from the definition of $S_{\infty}$ and $\mu_{i, \infty}$.
 
 Now suppose $M$ is an abelian monoid and that for each $i \in I$ we have identity-preserving morphisms $\psi:S_i \rightarrow M$ such that $\psi_i=\psi_j \circ \mu_{ij}$.  Define $\psi: S_{\infty} \rightarrow M$ as follows.  Let $[s]=\mu_{i, \infty }(s) \in S_{\infty}$ such that $s \in S_i$ for some $i \in I$.  Now define $\psi$ by $\psi([s])=\psi_i(s)$.  If $[s]=[t]$ in $S_{\infty}$ for some $t \in S_j$, then there exists $k \in I$ with $k \geq i, j$ such that $\mu_{ik}(s)=\mu_{jk}(t)$ and
 $$\psi_i(s)=\psi_k (\mu_{ik}(s))=\psi_k( \mu_{jk}(t))=\psi_j(t).$$
 So $\psi$ is well-defined.
 
 Note that for $s \in S_i$, we have $\psi (\mu_{i,\infty} (s))=\psi_i (s)$ by definition, so that $\psi_i=\psi \circ \mu_{i, \infty}$.  A final quick check shows that $\psi([0])=\psi( \mu_{i,\infty} (0))=\psi_i(0)=0$, so that $\psi$ is indeed an identity-preserving monoid morphism.
 \end{proof}

\begin{lemma}\label{l:iso-direct-limits-monoids}
Let $\mathcal{S} = \left( ( S_{i} )_{ i \in I } , ( \mu_{ij} )_{ i \leq j \text{ in } I} \right)$ and $\mathcal{T} = \left( ( T_{i} )_{ i \in I } , ( \nu_{ij} )_{ i \leq j \text{ in } I} \right)$ be direct systems in $\textbf{CMon}_{0}$.  Suppose there exists a collection of morphisms 
\begin{align*}
\left( \ftn{ \psi_{i} }{ S_{i} }{ T_{i} } \right)_{ i\in I}
\end{align*}
such that $\nu_{ij} \circ \psi_{i}  = \psi_{j} \circ \mu_{ij}$ for all $i \leq j$ in $I$.  Then there exists a unique morphism $\ftn{ \psi }{ S_{\infty} }{ T_{\infty} }$ such that 
\begin{align*}
\nu_{i, \infty } \circ \psi_{i} = \psi \circ \mu_{ i , \infty }
\end{align*}
for all $i \in I$.  Consequently, if $\psi_{i}$ are monoid isomorphisms for all $i \in I$, $\psi$ is a monoid isomorphism.
\end{lemma}

\begin{proof}
Let $i \leq j$ in $I$.  Note that 
\begin{align*}
\nu_{ i, \infty } \circ \psi_{i} = \nu_{ j, \infty } \circ \nu_{ij} \circ \psi_{i} = \nu_{ j , \infty } \circ \psi_{j} \circ \mu_{ij}.
\end{align*}
Hence, by Lemma~\ref{l:direct-limits-monoids}, there exists a morphism $\ftn{ \psi }{ S_{\infty} }{ T_{\infty} }$ such that 
\begin{align*}
\nu_{ i ,\infty } \circ \psi_{i} = \psi \circ \mu_{ i , \infty }.
\end{align*}

Suppose $\ftn{ \phi }{ S_{\infty} }{ T_{\infty} }$ is a morphism such that 
\begin{align*}
\nu_{ i ,\infty } \circ \psi_{i} = \phi \circ \mu_{ i , \infty }.
\end{align*}
Let $a \in S_{\infty}$.  By Lemma~\ref{l:direct-limits-monoids}, there exist $i \in I$ and $a_{1} \in S_{i}$ such that $a = \mu_{i, \infty } ( a_{1} )$.  Therefore, 
\begin{align*}
\psi ( a ) = \psi ( \mu_{ i, \infty } ( a_{1} ) ) = \nu_{ i ,\infty } \circ \psi_{i} ( a_{1} ) = \phi ( \mu_{ i , \infty } ( a_{1} ) ) = \phi ( a ).
\end{align*}
\end{proof}

\begin{definition}
Let $E$ and $F$ be arbitrary graphs.  By definition (see \cite{kg:leavittlimits}, p. 8), a graph morphism $\eta: E \rightarrow F$ is a \textit{CK-morphism} provided
 \begin{itemize}
 \item[(1)]  The restrictions $\eta^0$ to the vertex set $E^0$  and $\eta^1$ to edge set $E^1$  are both injective;
 \item[(2)]  For each $v \in E^0$ which is neither a sink nor an infinite emitter, $\eta^1$ induces a bijection $s_{E}^{-1}(v) \rightarrow s_{F}^{-1}(\eta^0(v))$.
 \end{itemize}
 Note that any CK-morphism must map infinite emitters to infinite emitters.  So $v \in E^0$ an infinite emitter implies $\eta(v) \in F^0$ is an infinite emitter.
 \end{definition}
 
Let \textbf{CKGr} be the category whose objects are directed graphs and whose set of morphisms are CK-morphisms.  By Lemma~2.5 of  \cite{kg:leavittlimits}, arbitrary direct limits exist in this category.  As in Section~2.2 of \cite{kg:leavittlimits}, given a field $K$, $K$-\textbf{Alg} will denote the category of (not necessarily unital) $K$-algebras.  Thus, objects in $K$-\textbf{Alg} are arbitrary $K$-algebras (that is, arbitrary vector spaces over $K$, equipped with an associative, $K$-bilinear multiplication), and sets of morphisms are arbitrary multiplicative $K$-linear maps.  It is well-known that direct limits exist in categories related to many algebraic structures (see \cite{faith:alg-reference}, section I.5).  In particular, given a directed system of $K$-algebras, a direct limit exists in the category of (not necessarily unital) rings.  It is a straightforward exercise to define a $K$-algebra structure on this direct limit.  So direct limits exist in $K$-\textbf{Alg}.


Making use of Lemma~2.5 of \cite{kg:leavittlimits}, let $\mathscr{E}=\left( (E_i)_{i \in I}, (\phi_{ij})_{i \leq j~\mbox{in }I} \right)$ be a direct system in \textbf{CKGr}.  Let $\mathscr{A}=\left( (A_i)_{i \in I}, (\gamma_{ij})_{i \leq j~\mbox{in }I} \right)$ be the corresponding direct system of Leavitt path algebras (so that $A_i=L_K(E_i)$ and $\gamma_{ij}=L_K(\phi_{ij})$).  Let $A_{\infty}$ be the direct limit of this system in $K$-\textbf{Alg} (so $A_{\infty}=L_K(E_{\infty})$).  Then we have maps $\iota_i=L_K(\eta_i)$ with $\iota_i:A_i \rightarrow A_{\infty}$ such that $\iota_i=\iota_j \circ \gamma_{ij}$ and where $\eta_i$ are CK-morphisms.

A direct system $\mathscr{A}=\left( (A_i)_{i \in I}, (\gamma_{ij})_{i \leq j~\mbox{in }I} \right)$ in $K$-\textbf{Alg} gives rise to a direct system
$\mathcal{V}[\mathscr{A}]=\left( (\mathcal{V}[A_i])_{i \in I}, (\mathcal{V}[\gamma_{ij}])_{i \leq j~\mbox{in }I} \right)$ in $\textbf{CMon}_0$.
By Lemma~\ref{l:direct-limits-monoids}, there exists an abelian monoid $\mathcal{V}_{\infty}$ that is the direct limit of this system.  So there are maps $\lambda_i: \mathcal{V}[A_i] \rightarrow \mathcal{V}_{\infty}$ such that 
$$\lambda_i= \lambda_j \circ \mathcal{V}[\gamma_{ij}].$$  
This leads to the following.

\begin{lemma}\label{l:Matrix-Entries}
Let $\mathscr{A}=\left( (A_i)_{i \in I}, (\gamma_{ij})_{i \leq j~\mbox{in }I} \right)$ be a direct system as above.  Then, for any $[e] \in \mathcal{V}[A_{\infty}]$, there exists a positive integer $n$ and $i \in I$ such that  $e \in M_n(\iota_i(A_i))$.  Moreover, the algebra homomorphism $\iota_i: A_i \rightarrow A_{\infty}$ induces an identity-preserving monoid morphism from $\mathcal{V}[A_i]$ to $\mathcal{V}[A_{\infty}]$.
\end{lemma}

\begin{proof}
Let $[e] \in \mathcal{V}[A_{\infty}]$.  By definition of $\mathcal{V}[A_{\infty}]$, we have that $e \in M_{\infty} (A_{\infty})$.  Since $ M_{\infty} (A_{\infty})=\bigcup_{n=1}^{\infty} M_{n} (A_{\infty})$, any element of $ M_{\infty} (A_{\infty})$ must belong to some $ M_{n} (A_{\infty})$.  So $e \in  M_{n} (A_{\infty})$ for some $n$.  

By definition of $A_{\infty}$, each element of $A_{\infty}$ is in the image of $\iota_k: A_k \rightarrow A_{\infty}$ for some $k$.  For each $(r,s)$-entry $e_{rs}$ of $e$,  let $\iota_{j(r,s)}$ be such a map.  Choose $i \geq j(r,s)$ for all $(r,s)$.  Since $\iota_i=\iota_{j(r,s)} \circ \gamma_{i j(r,s)}$, we have $a \in M_n(\iota_i(A_i))$.  

For the final claim, for each $n$ the map $\iota_i: A_i \rightarrow A_{\infty}$ clearly gives a family of homomorphisms $\iota_i^n: M_n( A_i) \rightarrow M_n(A_{\infty})$.  Note that for $f  \in M_n(A_i)$, $f \oplus 0 \in M_m(A_i)$ for $m>n$ and
$$\iota_i^m(f)=\iota_i^n(f) \oplus 0.$$
Therefore $\iota_i$ induces a well-defined map 
$$\mathcal{V}[\iota_i]: M_{\infty}(A_i) \rightarrow M_{\infty}(A_{\infty})$$
defined by $\mathcal{V}[\iota_i](f)=\iota_i^n(f)$ (where $f \in M_n(A_i) \subseteq M_{\infty}(A_i)$).

Now consider $[f]=[g]$ in $\mathcal{V}[A_i]$.  By definition $f=ab$ and $g=ba$ for some matrices $a, b \in M_{\infty}(A_i)$.  For some sufficiently large $n$ we have
$$\mathcal{V}[\iota_i](f)=\iota_i^n(f)=\iota_i^n(ab)=\iota_i^n(a)\iota_i^n(b)=
\mathcal{V}[\iota_i](a)\mathcal{V}[\iota_i](b),$$
using the fact that $\iota_i^n$ is a homomorphism.  A similar result shows $ 
\mathcal{V}[\iota_i](g)=\mathcal{V}[\iota_i](b)\mathcal{V}[\iota_i](a)$.  This shows 
$\mathcal{V}[\iota_i](f)$ is equivalent to  $\mathcal{V}[\iota_i](g)$ in $M_{\infty}(A_{\infty})$.  Clearly
$\mathcal{V}[\iota_i](0)=0$. Therefore 
$$\mathcal{V}[\iota_i]: \mathcal{V}(A_i) \rightarrow \mathcal{V}(A_{\infty})$$ is a well-defined identity-preserving monoid morphism.

\vspace{.2cm}
\end{proof}

\begin{lemma}\label{l:equivalent-idempotent}
Let $\left( ( A_{i} )_{i \in I } , ( \phi_{ij} )_{ i \leq j  \text{ in } I } \right)$ be a direct system of algebras and let $A_{\infty}$ be the direct limit.  Let $e, f$ be idempotents in $A_{i}$ such that $\phi_{i,\infty} ( e ) \sim \phi_{i, \infty } ( f )$ in $A_{\infty}$.  Then there exists $j \in I$ with $i \leq j$ and $\phi_{ij}(e) \sim \phi_{ij}(f)$ in $A_{i}$.
\end{lemma}

\begin{proof}
Let $x , y \in A_{\infty}$ such that $\phi_{i, \infty } ( e ) = x y$ and $\phi_{ i, \infty } ( f ) = y x$.  Then there exists $k \in I$ with $i \leq k$ and there exist $s, t \in A_{k}$ such that $\phi_{k, \infty } ( s ) = x$ and $\phi_{ k, \infty } ( t ) = y$.  Hence, 
\begin{align*}
\phi_{i, \infty } ( e ) = \phi_{ k , \infty } ( s t ) \quad \text{and} \quad \phi_{ i, \infty } ( f ) = \phi_{ k, \infty } ( t s ).
\end{align*}    
Hence, there exists $j \in I$ with $i \leq j$ and $k \leq j$ such that 
\begin{align*}
\phi_{ ij } ( e ) = \phi_{ kj } ( s t )  \quad \text{and} \quad \phi_{ ij} (f) = \phi_{ kj} ( t s ).
\end{align*}
Therefore, $\phi_{ij} ( e ) \sim \phi_{ij} ( f )$ in $A_{j}$.
\end{proof}

\begin{lemma}\label{l:non-stable-iso}
Let $\left( ( E_{i} )_{i \in I } , ( \phi_{ij} )_{ i \leq j \text{ in } I} \right)$ be a direct system in \textbf{CKGr} with corresponding object $E_{\infty}$ and let 
\begin{align*}
\left(  ( L_{K} ( E_{i} ) )_{ i \in I } , ( L_{K} ( \phi_{ij} ) )_{ i \leq j \text{ in } I } \right)
\end{align*}
 be the direct system in $K$-\textbf{Alg}.  Set $\nu_{ij} =  \mathcal{V} [  L_{K} ( \phi_{ij} ) ]$.  If $\mathcal{V}_{\infty}$ is the direct limit of 
\begin{align*}
\left(  ( \mathcal{V} [ L_{K} ( E_{i} ) ]  )_{ i \in I } , (  \nu_{ij} )_{ i \leq j \text{ in } I } \right)
\end{align*}
in $\textbf{CMon}_{0}$, there exists a monoid isomorphism $\ftn{ \psi_{ \mathcal{V} } }{ \mathcal{V}_{\infty} }{ \mathcal{V} [ L_{K} ( E_{\infty} ) ] }$ such that 
\begin{align*}
\psi_{ \mathcal{V} } \circ \nu_{i, \infty } = \mathcal{V} [ L_{K} ( \phi_{i, \infty } ) ].
\end{align*}
\end{lemma}

\begin{proof}
The algebra homomorphisms $\gamma_{ij}: A_i \rightarrow A_j$ and $\iota_i: A_i \rightarrow A_{\infty}$ induce maps $$\mathcal{V}[A_i] \overset{\mathcal{V}[\gamma_{ij} ]}{\longrightarrow} \mathcal{V}[A_j]
\hspace{.5cm}
\mbox{and}
\hspace{.5cm}
\mathcal{V}[A_i] \overset{\mathcal{V}[\gamma_{ij} ]}{\longrightarrow} \mathcal{V}[A_{\infty}]$$
so that, for each $i \in I$, we have that $\mathcal{V}[\iota_i]: \mathcal{V}[A_i] \rightarrow \mathcal{V}[A_{\infty}]$ with the property 
$$\mathcal{V}[\iota_i]=\mathcal{V}[\iota_j] \circ \mathcal{V}[\gamma_{ij}].$$  By the universal property of $\mathcal{V}_{\infty}$, there exists a monoid morphism $\psi: \mathcal{V}_{\infty} \rightarrow \mathcal{V}[A_{\infty}]$ such that
$$\psi \circ \nu_{i, \infty}=\mathcal{V}[\iota_i].$$
We now show that $\psi$ is a monoid isomorphism.

Let $[e] \in \mathcal{V}[A_{\infty}]$.  By Lemma~\ref{l:Matrix-Entries}, $e \in M_n(\iota_i (A_i))$ for some integer $n$ and $i \in I$.  Let $e_{rs}$ denote the $(r,s)$-entry of the matrix $e$.  Then there exists a matrix $f \in M_n (A_i)$ with $(r,s)$-entry $f_{rs}$ such that $\iota_i (f_{rs})=e_{rs}$.   
Then we have 
$$e=\mathcal{V}[\iota_i](f).$$
Since $\mathcal{V}[\iota_i]=\psi \circ \nu_{i, \infty}$, we have that
$$[e]=\psi( \nu_{i, \infty} ([f]))$$
so that $\psi$ is surjective.

By Lemma~\ref{l:equivalent-idempotent}, $\psi$ is injective.  Therefore, $\psi$ is a monoid isomorphism.

%

\end{proof}

We want to show that $\overline{M} ( - )$ is a continuous functor.  Let $\left( ( E_{i} )_{i \in I } , ( \phi_{ij} )_{ i \leq j \text{ in } I} \right)$ be a direct system in \textbf{CKGr}.  For clarity, we write $a_v, ~a_{v, S}$ for elements of $\overline{M}_{E_k}$ and $b_w, ~b_{w, T}$ for elements of $\overline{M}_{E_{\infty}}$.  Note that any CK-morphism sends regular vertices to regular vertices (and infinite emitters to infinite emitters).  Therefore, if $w \in E^0_{\infty}$ is a regular vertex, then there is some $i \in I$ and $v \in E^0_i$ such that $w=\phi^0_{i,\infty}(v)$ and $|s^{-1}_{E_i}(v)|=|s^{-1}_{E_{\infty}}(w)|$.  We often write $b_w=b_{\phi_{i, \infty}(v)}$ or just $ b_w=b_v$ for this situation.

Similarly, for $w \in E^0_{\infty}$ an infinite emitter and $T$ a finite subset of $s^{-1}_{E_{\infty}}(w)$, there is $i \in I$ and $v \in E^0_i$ where $w=\phi^0_{i,\infty}(v)$ and, for every $f \in T$, there is $e \in s^{-1}_{E_i}(v)$ with $f=\phi^1_{i, \infty}(e)$.  For notation convenience, we often take $S$ to be the set of all such $e \in s^{-1}_{E_i}(v)$ and write $b_{w, T}=b_{\phi_{i,\infty}(v), \phi_{i,\infty}(S)}$ or just $b_{w,T}=b_{v, S}$.

\begin{lemma}\label{l:M-isomorphism}
Let $\left( ( E_{i} )_{i \in I } , ( \phi_{ij} )_{ i \leq j \text{ in } I} \right)$ be a direct system in \textbf{CKGr} and let $(E_{\infty},\phi_i)$ be the direct limit of this system.  Let $\left( ( \overline{M}_{E_{i}} )_{ i \in I } , ( \overline{M} ( \phi_{ij} ) )_{ i , j \text{ in } I } \right)$ be the direct system in $\textbf{CMon}_{0}$ and let $(\overline{M}_{\infty}, \mu_{i,\infty})$ be the direct limit of this system.  Set $\mu_{ij} = \overline{M}(\phi_{ij})$.  Then there exists a monoid isomorphism $\ftn{ \psi_{ \overline{M} } }{ \overline{M}_{\infty} }{ \overline{M}_{ E_{ \infty } } }$ such that 
\begin{align*}
\psi_{ \overline{M} } \circ \mu_{i,\infty} = \overline{M} ( \phi_{i, \infty } ).
\end{align*}
\end{lemma}

\begin{proof}
For any $i\in I$, the CK-morphism $\phi_{i, \infty}$ induces a monoid morphism $\overline{M}(\phi_{i, \infty}) : \overline{M}_{E_i}  \rightarrow  \overline{M}_{ E_{ \infty } } $ given by
$$\overline{M}(\phi_{i, \infty})(a_v)=b_{\phi_{i, \infty}(v)}~~~~\mbox{and}~~~~\overline{M}(\phi_{i, \infty})(a_{v,S})=b_{\phi_{i, \infty}(v), \phi_{i, \infty}(S)}.$$

It follows easily that $\overline{M}(\phi_{i,\infty}) =  \overline{M}(\phi_{j, \infty}) \circ \mu_{ij}$.  That there exists a surjective monoid morphism $\ftn{ \psi_{ \overline{M} } }{ \overline{M}_{\infty} }{ \overline{M}_{ E_{ \infty } } }$ is seen from the definition of $\overline{M}_{\infty}$ as the direct limit of the system $\left( ( \overline{M}_{E_{i}} )_{ i \in I } , ( \overline{M} ( \phi_{ij} ) )_{ i , j \text{ in } I } \right)$.  

Showing that $ \psi_{ \overline{M} }$ is injective is the remainder of the proof.  Our strategy is as follows.  We define a monoid morphism $\theta: \overline{M}_{0} 
\rightarrow \overline{M}_{\infty}$ on the free abelian monoid $\overline{M}_0$ generated by the same generators of $\overline{M}_{E_\infty}$, but without any relations.  We then use $\theta$ to define
a monoid morphism $\overline{\theta}: \overline{M}_{ E_{ \infty } } \rightarrow \overline{M}_{\infty}$ such that $\overline{\theta} \circ \psi_{\overline{M}}$ is the identity map on $\overline{M}_{\infty}$.  It then follows that $\psi_{ \overline{M}}$ is injective

We define $\theta$ on $\overline{M}_0$ as follows.  For a generator of the form $b_w \in \overline{M}_{0}$ with $w \in E^0_{\infty}$,
let 
$$\theta(b_w) = \mu_{i,\infty}(a_v),$$ 
where $w = \phi_{i,\infty}(v)$ for some $i \in I$ and $v \in E_i^0$.

For a generator of the form $b_{w,T} \in \overline{M}_{0}$, with $w \in E^0_{\infty}$, and $T$ a finite subset of $s^{-1}_{E_{\infty}} (w)$, there is an $i \in I$ 
with $v \in E^0_i$ and a finite subset $S \subseteq s^{-1}_{E_i}(v)$ with $\phi_{i, \infty} (S)=T$. Define the map $\theta$ on $b_{w,T}$ by
$$\theta(b_{w,T})=\mu_{i, \infty} (a_{v, S}).$$

We now show that ${\theta}$ is well-defined.

Since there are no relations on $\overline{M}_0$, we only need to show that $\theta$ does not depend on our choice of $i \in I$.  First consider the generator $b_w \in \overline{M}_{0}$.  Suppose $w=\phi_{i,\infty}(v) = \phi_{j,\infty}(v')$ for some $i, j \in I$, $v \in E_i^0$ and $v' \in E_j^0$.  Without loss of generality, we may assume $i \le j$.  Since $\phi_{i,\infty}(v) = \phi_{j,\infty}(v')$, 
we must have $\phi^0_{ij}(v) = v'$, so that $a_{\phi_{ij}(v)} = a_{v'}$.  
Thus, $\mu_{ij}(a_v) = \overline{M}(\phi_{ij})(a_v) = a_{v'}$ and it follows that $\mu_{i,\infty}(a_v) = \mu_{j,\infty}(a_{v'})$. 

 A similar approach works for the generator $b_{w, T} \in \overline{M}_{0}$.  Suppose $w=\phi_{i,\infty}(v) = \phi_{j,\infty}(v')$ and $T=\phi_{i,\infty}(S)=\phi_{j,\infty}(S')$  for some $i, j \in I$, $v \in E_i^0$ and $v' \in E_j^0$.  Without loss of generality, we may assume $i \le j$.  Since $\phi_{i,\infty}(v) = \phi_{j,\infty}(v')$, 
we must have $\phi^0_{ij}(v) = v'$.  Similarly, $\phi^1_{i,\infty}(S)=\phi^1_{j,\infty}(S')$, so for every $e \in S$ we have $e' \in S'$ such that
$\phi^1_{ij}(e) =e'$.  We denote this succinctly by $\phi_{ij}(S) = S'$.  Then $\overline{M}(\phi_{ij})(a_{v, S})=a_{\phi_{ij}(v), \phi_{ij}(S)} = a_{v', S'}$ and it follows that $\mu_{i,\infty}(a_{v,S}) = \mu_{j,\infty}(a_{v', S'})$.  So the map $\theta$ does not depend upon a choice of $i \in I$.

We now see that $\theta$ is well-defined on the generators.  Now extend additively to give a map $\overline{M}_{0} \rightarrow \overline{M}_{\infty}$, e.g.,
\begin{align*}
  {\theta} \left(\sum b_{w_{j}} +\sum b_{w_{k}, T_k} \right)&=
\sum \theta(b_{w_{j}}) +\sum \theta(b_{w_{k}, T_k}) \\ \\
&=
\sum \mu_{i,\infty}(a_{v_{j}}) +\sum \mu_{i, \infty} (a_{v_{k}, S_k}).
\end{align*}

Our next goal is to essentially quotient out by the kernel of $\theta$ to obtain a monoid morphism $\overline{\theta}$ on $\overline{M}_{E_\infty}$.  For this to work we need to
show that the kernel of $\theta$ contains our relations on $\overline{M}_{E_\infty}$.  To this end,
let $R$ denote the set of relations (defined on the generators) distinguishing $\overline{M}_{E_\infty}$ from $\overline{M}_0$, and let $\rho$ denote the equivalence relation 
generated by $R$,  so that $\overline{M}_{E_\infty} \cong \overline{M}_0/\rho$.  Note that one typically does not distinguish between $R$ and $\rho$, but we are being overly
cautious in our treatment here.  When thinking of these as ordered pairs, technically, we have $R \subseteq \rho$.  Let $Q$ be the natural quotient map taking a word $x$ in $\overline{M}_0$ to its equivalence class in $\overline{M}_{E_\infty}$.  Thinking of $\ker \theta$ as a collection of ordererd pairs $(x,y)$ where $\theta(x) = \theta (y)$, we now wish to show that $\rho \subseteq \ker \theta$.

Let $x$ and $y$ be words in $\overline{M}_0$ such that $x = y$ in $\overline{M}_{E_\infty}$, that is $(x,y) \in \rho$, or equivalently, $Q(x) = Q(y)$.  Then there
exists a finite sequence of words $x_1, \dots ,x_{k+1}$ in $\overline{M}_0$ with $x_1 = x$ and $x_{k+1} = y$ such that
$x_{i+1}$ is obtained from $x_i$ by substituting a term $z_i$ of $x_i$ by $y_i$ for some $y_i$ in $\overline{M}_0$ such that ``$z_i = y_i$''
is one of the relations in $R$ ($(z_i,y_i) \in R$).  Now, if $\theta(z_i) = \theta(y_i)$  for every such pair in $R$, then by transitivity we will have $\theta(x) = \theta(y)$.
Thus, to conclude that $\rho \subseteq \ker \theta$, it now suffices to show that $\theta$ respects the three forms of relations found in $R$.  

We first consider relations in $R$ of the form
$b_w=\sum_{f \in s^{-1}_{E_{\infty}}(w)} b_{r (f)}$ for a regular vertex $w \in E^0_{\infty}$. In this case, we have $w =\phi_{i, \infty}(v)$ for some $v \in E^0_i$ and
$\theta(b_w)=\mu_{i, \infty}(a_v)$.  For some choice $\ell \in I$ we have $w=\phi_{\ell, \infty}(v')$ and $|s^{-1}_{E_{\infty}}(w)|=|s^{-1}_{E_{\ell}}(v')|$ so that
\begin{align*}
{\theta} \left(  \sum_{f \in s^{-1}_{E_{\infty}}(w)} b_{r (f)}\right)&=
\sum_{f \in s^{-1}_{E_{\infty}}(w)} \theta(b_{r (f)})=
\sum_{e \in s^{-1}_{E_{\ell}}(v')} \mu_{\ell, \infty} (a_{r (e)} ).
\end{align*}
Since $\mu_{\ell, \infty}$ is a monoid morphism and $v'$ a regular vertex, this in turn equals
\begin{align*}
 \mu_{\ell, \infty} \left( \sum_{e \in s^{-1}_{E_{\ell}}(v')} a_{r (e)} \right)=
  \mu_{\ell, \infty} (a_{v'} ).
\end{align*}
As argued previously, we must have $  \mu_{\ell, \infty} (a_{v'} )=\mu_{i, \infty}(a_v)$, so ${\theta}$ respects the relation defined on finite emitters.

Next consider the relations on infinite emitters in $R$ of the form $b_w=b_{w, T} +\sum_{f \in T} b_{r(f)}$ for an infinite emitter $w$.
Since CK-morphisms map infinite emitters to infinite emitters, we must have some $i \in I$ and infinite emitter $v \in E^0_i$ such that ${\theta}(b_w)=\mu_{i, \infty}(a_v)$.
Now consider ${\theta}\left(b_{w, T} +\sum_{f \in T} b_{r (f)}\right)=
\theta(b_{w, T}) +\sum_{f \in T} \theta(b_{r (f)})$.  As before, there is an $\ell \in I$ and an infinite emitter $v' \in E^0_{\ell}$ with $w=\phi_{\ell, \infty}(v')$.  Since $T$ is finite, we also must have a finite set $S \subseteq s^{-1}_{E_{\ell}}(v')$ such that $T=\phi^1_{\ell, \infty}(S)$.  Then

\begin{align*}
\theta(b_{w, T}) +\sum_{f \in T} \theta(b_{r (f)})&=
\mu_{\ell, \infty} (a_{v', S}) +\sum_{e \in S} \mu_{\ell, \infty}(a_{r (e)})=
\mu_{\ell, \infty} \left(a_{v', S}+\sum_{e \in S}a_{r (e)} \right).
\end{align*}
Since $v \in E^0_{\ell}$ is an infinite emitter with finite subset $S \subseteq s^{-1}_{E_{\ell}}(v')$, we must have $a_{v', S}+\sum_{e \in S}a_{r (e)}=a_{v'}$ and 
$${\theta} \left(b_{w, T} +\sum_{f \in T} b_{r (f)} \right)=\mu_{\ell, \infty} (a_{v'}).$$
Again, because ${\theta}$ does not depend upon the choice of index, we must have
$${\theta}(b_w)={\theta} \left(b_{w, T} +\sum_{f \in T} b_{r (f)} \right).$$

Lastly, we consider the relations in $R$ of the form
$$b_{w,T} + \sum_{ f \in T \setminus T' }  b_{ r_{E_\infty} (f) }= b_{w,T'} + \sum_{ f \in T'\setminus T } b_{ r_{E_\infty} (f) }.$$
Again, since CK-morphisms map infinite emitters to infinite emitters, there is some $i \in I$ and infinite emitter $v \in E^0_i$ such that ${\theta}(b_w)=\mu_{i, \infty}(a_v)$.
Now consider 
\begin{align*}
{\theta}\left(b_{w, T} +\sum_{f \in T \setminus T'} b_{r (f)}\right)=
\theta(b_{w, T}) +\sum_{f \in T \setminus T'} \theta(b_{r (f)}).
\end{align*}
As before, there is an $\ell \in I$ and an infinite emitter $v' \in E^0_{\ell}$ with $w=\phi_{\ell, \infty}(v')$.  Since $T$ and $T'$ are finite, we also must have finite sets $S,S' \subseteq s^{-1}_{E_{\ell}}(v')$ such that $T=\phi^1_{\ell, \infty}(S)$ and $T'=\phi^1_{\ell, \infty}(S')$.  
Then
\begin{align*}
\theta(b_{w, T}) +\sum_{f \in T \setminus T'} \theta(b_{r (f)})&=
\mu_{\ell, \infty} (a_{v', S}) +\sum_{e \in S \setminus S'} \mu_{\ell, \infty}(a_{r (e)}) =
\mu_{\ell, \infty} \left(a_{v', S}+\sum_{e \in S \setminus S'}a_{r (e)} \right) \\ \\
&=
\mu_{\ell, \infty} \left(a_{v', S}+\sum_{e \in S' \setminus S}a_{r (e)} \right) = 
\theta(b_{w, T}) +\sum_{f \in T' \setminus T} \theta(b_{r (f)}).
\end{align*}
Once again, we see that this relation is preserved by $\theta$ and we conclude that indeed $\rho \subseteq \ker \theta$.


Finally we can define $\overline{\theta}:\overline{M}_{E_\infty} \rightarrow \overline{M_\infty}$ by $\overline{\theta}(Q(x)) = \theta(x).$  This is well-defined since if $Q(x) = Q(y)$, we have
$(x,y) \in \rho \subseteq \ker \theta$ so that $\theta(x) = \theta(y)$.

Now, if $\overline{\theta} \circ \psi_{\overline{M}}$ equals the identity map on $\overline{M}_{\infty}$, 
then $\psi_{\overline{M}}$ is injective.  Since $\overline{M}_{\infty}$ is generated by the set
$$\left\{\mu_{i,\infty}(a_v):v \in E_i^0\} \cup \{\mu_{i,\infty}(a_{v,S}): v \mbox{ an inf.
 emitter, fin. non-empty }S \subseteq s_{E_i}^{-1}(v)\right\},$$
we need only check the above condition on this set of generators.  In this direction, suppose that
$\mu_{i,\infty}(a_v) \in \overline{M}_{\infty}$, where $v \in E_i$ for some $i \in I$.  Then
\begin{eqnarray*}\overline{\theta}(\psi_{\overline{M}}(\mu_{i,\infty}(a_v))  =  \overline{\theta}(\overline{M}(\phi_{i,\infty})(a_v)) =  \overline{\theta}(b_{\phi_{i,\infty}(v)})  =  \mu_{i,\infty}(a_v).
\end{eqnarray*}

Similarly,
suppose that
$\mu_{i,\infty}(a_{v,S}) \in \overline{M}_{\infty}$, where $v \in E_i$ for some $i \in I$ and $S$ is a finite non-empty
subset of $s_{E_i}^{-1} ( v )$.  Then
\begin{eqnarray*}\overline{\theta}(\psi_{\overline{M}}(\mu_{i,\infty}(a_{v,S})) &  = & \overline{\theta}(\overline{M}(\phi_{i,\infty})(a_{v,S})) =  \overline{\theta}(b_{\phi_{i,\infty}(v), \phi_{i,\infty}(S)})  =  \mu_{i,\infty}(a_{v,S}).
\end{eqnarray*}
Thus, $\psi_{\overline{M}}$ is injective and $\psi_{\overline{M}}: \overline{M}_{\infty} \longrightarrow \overline{M}_{E_{\infty}}$ is a monoid isomorphism.
\end{proof}

 \begin{lemma}\label{l:commutative-CK-morphism}
 Let $\eta: E \rightarrow F$ be a CK-morphism.  For the maps ${\gamma}_E$, ${\gamma}_F$ defined in Lemma~\ref{l:naturalhom}, and the maps $\overline{M}[\eta]$ and $(\mathcal{V} \circ L_K) [\eta ]$ induced by the functors $\overline{M}$ and $(\mathcal{V} \circ L_K)$, respectively, the diagram
\begin{align*}
\xymatrix{
\overline{M}_E \ar[r]^-{ {\gamma}_E } \ar[d]_{ \overline{M}[\eta] } & \mathcal{V}[ L_K(E)] \ar[d]^{(\mathcal{V} \circ L_K) [\eta]} \\
\overline{M}_F \ar[r]_-{ {\gamma}_F } &  \mathcal{V}[ L_K(F)]
}
\end{align*}
is commutative.
\end{lemma}

\begin{proof}
For graphs $E=(E^0, ~E^1, ~r_E,~s_E)$ and $F=(F^0, ~F^1, ~r_F,~s_F)$, let $\{p_v, t_e, t_e^*:~~v \in E^0,~e \in E^1\}$ and $\{q_w, u_f, u_f^*:~~w \in F^0,~f \in F^1\}$ be the Leavitt $E$-family and $F$-family, respectively, generating $L_K(E)$ and $L_K(F)$, respectively.

Similarly, suppose $\overline{M}_E$ and $\overline{M}_F$ are generated by elements
 $$\{ a_v :~v \in E^{0} \} 
 \cup \{a_{v, S}:~v \in E^0~\mbox{$v$ an inf. emitter, finite $\emptyset \neq S \subseteq s_E^{-1}(v)$} \}$$
 and
 $$\{ b_w :~w \in F^0\} 
 \cup \{b_{w,T}:~v \in F^0~\mbox{$w$ an inf. emitter, finite $\emptyset \neq T \subseteq s_F^{-1}(w)$} \},$$
respectively. 

Let $v \in E^0$ be a non-singular vertex.  By the definition of ${\gamma}$, ${\gamma}_E ([a_v])=[p_v]$.  Then by Lemma~2.5 of \cite{kg:leavittlimits},  using the functor $(\mathcal{V} \circ L_K)$, we have
$$(\mathcal{V} \circ L_K) [\eta] ([ p_v ] )= [ q_{\eta(v)} ].$$
Hence, $(\mathcal{V} \circ L_K) [\eta]) \circ {\gamma}_E ([a_v])=[q_{\eta(v)}]$ for $v$ non-singular.

Similarly, the functor $\overline{M}$ induces the map $\overline{M}[\eta]$ defined for non-singular $v$ by $\overline{M}[\eta]([a_v])=[b_{\eta(v)}]$.  Then since
$$\overline{\gamma}_F ([b_{\eta(v)}])= [ q_{\eta(v)} ],$$
we have $((\mathcal{V} \circ L_K) [\eta] \circ \overline{\gamma}_E) ([a_v])=(\overline{\gamma}_F  \circ \overline{M}[\eta])([a_v])$ for non-singular $v \in E^0$.

Now assume $v \in E^0$ is an infinite emitter.  Since $\eta:E \rightarrow F$ is a CK-morphism, we also have $\eta(v) \in F^0$ is an infinite emitter.  For each element $e_{i_k}$ of a finite subset $S$ of  $s_E^{-1}(v)$, let $f_{i_k}=\eta^1(e_{i_k})$ and set
$T_{v, S}=\{f_{i_1}, \ldots, f_{i_n}\}$.  Note that the cardinality of $T_{v, S}$ equals that of $S$ because of property (1) of CK-morphisms.

By the definition of ${\gamma}_{E}$, for an infinite emitter $v \in E^0$ and $S$, a non-empty finite subset of $s_E^{-1} (v)$, we have
${\gamma}_E ([a_{v, S}])=[p_v-\sum_{e \in S} t_e t_e^*]$.  Again by Lemma~2.5 of \cite{kg:leavittlimits},
$$(\mathcal{V} \circ L_K) [\eta] \left( p_v-\sum_{e \in S} t_e t_e^*\right)=
q_{\eta(v)}-\sum_{\eta(e) \in \eta(S)} u_{\eta(e)} u_{\eta(e)}^*.$$
However, $\eta(e)=f \in T_{ v, S}$.
Hence,
$$(\mathcal{V} \circ L_K) [\eta]  \circ {\gamma}_E ([a_{v, S}])=\left[q_{\eta(v)}-\sum_{f \in T_{v,S } } u_{f} u_{f}^* \right].$$

Also, for $v \in E^0$ an infinite emitter and $S$ a non-empty finite subset of $s_E^{-1} (v)$, the induced map $\overline{M}[\eta]$ is defined by
$$\overline{M}[\eta] ([a_{v,S}])=[b_{w, T}]$$
where $w=\eta(v)$ and $T=\{f_{i_1}, \ldots f_{i_n} \}$ is the finite subset of $s_F^{-1}(w)$ such that 
$f_{i_k}=\eta(e_{i_k})$.  So $\overline{M}_E ([a_{v, S}])=[b_{\eta(v), T_{S,v} }]$.  Then by the definition of ${\gamma}_{F}$
$${\gamma}_F ([b_{\eta(v), T_{S,v}}])=
\left[ q_{\eta(v)}-\sum_{f \in T_{S,v} } u_f u_f^* \right].$$
It follows then that $(\mathcal{V} \circ L_K) [\eta]  \circ {\gamma}_E={\gamma}_F \circ \overline{M}[\eta]$ for all generators of $\overline{M}_E$.  Hence the diagram above commutes.

\end{proof}

We are now ready to use the results of the previous section to show that the monoid morphism $\gamma_{E}$ given in Lemma~\ref{l:naturalhom} is a monoid isomorphism for an arbitrary graph.  

\begin{theorem}\label{t:iso-arbitary}
Let $E$ be an arbitrary graph and let $K$ be a field.  Then $\ftn{ \gamma_{E} }{ \overline{M}_{E} }{ \mathcal{V} [ L_{K} (E) ] }$ is a monoid isomorphism.
\end{theorem}

\begin{proof}
By Proposition~2.7 of \cite{kg:leavittlimits}, there exists a direct system $( (E_{i})_{ i \in I } , ( \phi_{ij} )_{ i ,j \in I } )$ in \textbf{CKGr} such that $E = \varinjlim ( E_{i} , \phi_{i,j} )$, $E_{i}$ are countable graphs, and 
\begin{align*}
L_{K} (E) = \varinjlim ( L_{K} ( E_{i} ) , L_{K} ( \phi_{ij} ) ).
\end{align*}
By Lemma~\ref{l:commutative-CK-morphism}, for each $i,j \in I$ with $i \leq j$, the diagram
\begin{align*}
\xymatrix{
\overline{M}_{E_{i}} \ar[d]_-{ \gamma_{E_{i}} } \ar[rr]^-{ \overline{M} [ \phi_{ij} ] }  & &\overline{M}_{E_{j} } \ar[d]^-{\gamma_{ E_{j} } }  \\
\mathcal{V} [ L_{K} ( E_{i} ) ] \ar[rr]_{ \mathcal{V} \circ L_{K} [ \phi_{ij} ] } & & \mathcal{V} [ L_{K} ( E_{j} ) ]
} 
\end{align*}
is commutative.  By Theorem~\ref{t:iso-countable}, $\gamma_{E_{i}}$ is a monoid isomorphism for all $i \in I$.  By Lemma~\ref{l:iso-direct-limits-monoids}, there exists a monoid isomorphism 
$\ftn{ \psi }{ \overline{M}_{\infty} }{ \mathcal{V}_{\infty} }$ such that $\nu_{i, \infty } \circ \gamma_{ E_{i} } = \psi \circ \mu_{ i , \infty }$,
where $\nu_{ij} = \mathcal{V} [ L_{K} (\phi_{ij}) ]$ and $\mu_{ ij } = \overline{M} ( \phi_{ij} )$.  

By Lemma~\ref{l:non-stable-iso} and Lemma~\ref{l:M-isomorphism}, there exist monoid isomorphisms
\begin{align*}
\ftn{\psi_{ \mathcal{V} } }{ \mathcal{V}_{ \infty } }{ \mathcal{V} [ L_{K} ( E) ] } \quad \text{and} \quad
\ftn{ \psi_{ \overline{M} } }{ \overline{M}_{\infty} }{ \overline{M}_{E} }
\end{align*} 
such that 
\begin{align*}
\psi_{ \mathcal{V} } \circ \nu_{ i, \infty } = \mathcal{V} [ L_{K} (\phi_{i,\infty}) ] \quad \text{and} \quad
\psi_{ \overline{M} } \circ \mu_{i, \infty } = \overline{M} ( \phi_{i, \infty } ).
\end{align*}

We claim that $\gamma_{E} = \psi_{ \mathcal{V} } \circ \psi \circ \psi_{ \overline{M} }^{-1}$.  First note that by Lemma~\ref{l:commutative-CK-morphism}, the diagram
\begin{align*}
\xymatrix{
\overline{M}_{E_{i}} \ar[rr]^-{ \overline{M} [ \phi_{i,\infty} ] } \ar[d]_{\gamma_{E_{i} } } & & \overline{M}_{E} \ar[d]^{ \gamma_{E}}\\
\mathcal{V} [ L_{K} ( E_{i} ) ] \ar[rr]_-{ \mathcal{V} [ L_{K} ( \phi_{i, \infty  } ) ]  } & & \mathcal{V} [ L_{K} ( E) ]
}
\end{align*}
is commutative.  Let $a \in \overline{M}_{E_{\infty}}$.  Then there exists $a_{1} \in \overline{M}_{E_{i}}$ such that $\mu_{i, \infty } ( a_{1} ) = \psi_{ \overline{M} }^{-1} ( a )$.  Then 
\begin{align*}
\psi_{ \mathcal{V} } \circ \psi \circ \psi_{ \overline{M} }^{-1} ( a ) &= \psi_{ \mathcal{V} }^{-1} \circ \psi \circ \mu_{ i, \infty } ( a_{1} ) = \psi_{ \mathcal{V} } \circ \nu_{ i , \infty } \circ \gamma_{E_{i}} ( a_{1} ) =  \mathcal{V} [ L_{K} (\phi_{i,\infty}) ] \circ \gamma_{E_{i}} ( a_{1} ) \\
	&= \gamma_{ E } \circ \overline{M}_{\phi_{i, \infty } } ( a_{1} ) = \gamma_{ E } \circ \psi_{\overline{M} } \circ \mu_{i, \infty } ( a_{1} ) = \gamma_{E} \circ \psi_{\overline{M} } \circ  \psi_{ \overline{M} }^{-1} ( a ) = \gamma_{E}(a).
\end{align*}
We have just proved the claim.  Since $\gamma_{E} = \psi_{ \mathcal{V} } \circ \psi \circ \psi_{ \overline{M} }^{-1}$ is the composition of monoid isomorphisms, $\gamma_{E}$ is a monoid isomorphism.
\end{proof}

\def\cprime{$'$}

\end{document}